\documentclass[11pt,reqno]{amsart}
 \usepackage{esint}
\usepackage{bbm}
\usepackage{}
\usepackage{amssymb}
\usepackage{mathrsfs}
\usepackage{amsfonts}
\usepackage{amsfonts,amssymb,amsmath,amsthm}
\usepackage{url}
\usepackage{enumerate}
\usepackage[pdftex,bookmarksnumbered]{hyperref}
\usepackage{graphicx,xcolor}
\newcommand{\ds}{\displaystyle}
\newtheorem{theorem}{Theorem}[section]
\newtheorem{lemma}{Lemma}[section]
\newtheorem{proposition}{Proposition}[section]
\newtheorem{corollary}{Corollary}[section]
\theoremstyle{definition}
\newtheorem{definition}{Definition}[section]
\newtheorem{remark}{Remark}[section]
\numberwithin{equation}{section}

\newtheorem{example}{Example}[section]

\usepackage[left=3.0cm, right=3.0cm, top=2.5cm, bottom=3.0cm]{geometry}

\usepackage{graphicx}%
\usepackage{color}
\usepackage{cite}
\usepackage{hyperref}
\hypersetup{
	colorlinks = true,%
	citecolor = blue,
	filecolor=red,%
	linkcolor = [rgb]{0.65,0.0,0.0},%
	anchorcolor = red,
	pagecolor = red,
	urlcolor= [rgb]{0.65,0.0,0.0}
	linktocpage=true,
	pdfpagelabels=true,
	bookmarksnumbered=true,
}

\DeclareMathOperator{\vol}{vol}

\DeclareMathOperator{\loc}{loc}

\DeclareMathOperator{\dd}{d}

\DeclareMathOperator{\supp}{supp}

\DeclareMathOperator{\spa}{span}

\DeclareMathOperator{\m}{\mathfrak{m}}
\DeclareMathOperator{\dm}{d\mathfrak{m}}

\DeclareMathOperator{\B}{\mathsf{B}}

\DeclareMathOperator{\FMMM}{\mathsf{FMMM}}
\newcommand{\Lip}{\mathsf{Lip}}
\newcommand{\dil}{\mathsf{dil}}

\newcommand{\Div}{{\text{div}}}

\author{Alexandru Krist\'aly}
\address{Department of Economics\\
	Babe\c s-Bolyai University\\
	400591 Cluj-Napoca, Romania \&  Institute of Applied Mathematics\\
 \'Obuda University\\
 1034 Budapest, Hungary}
  \email{alex.kristaly@econ.ubbcluj.ro; kristaly.alexandru@nik.uni-obuda.hu}

\author{Benling Li}
\address{
School of Mathematics and Statistics\\
Ningbo University\\
315211 Ningbo, China}
\email{libenling@nbu.edu.cn}

\author{Wei Zhao}
\address{
School of Mathematics\\
East China University of Science and Technology\\
200237 Shanghai, China}
\email{szhao\underline{ }wei@yahoo.com}


\thanks{A. Krist\'aly is  supported by the
	Excellence Researcher Program \'OE-KP-2-2022 of \'Obuda University, Hungary.  B. Li is supported by Natural Science Foundation of Zhejiang Province (No. LY23A010012) and Natural Science Foundation of Ningbo (No. 2024J017).  W. Zhao
is supported by Natural Science Foundation of China (No. 12471045) and Science and Technology Project of Xinjiang Production and Construction Corps (No. 2023CB008-12). }
\begin{document}

\title[]{Failure of famous functional inequalities on Finsler manifolds: the influence of $S$-curvature}
\vspace{0.5cm}
\begin{abstract}
The validity of functional inequalities on Finsler metric measure manifolds is based on three non-Riemannian quantities, namely, the \textit{reversibility}, \textit{ flag curvature} and $S$-\textit{curvature} induced by the measure. Under mild assumptions on the reversibility  and flag curvature, it turned out that famous functional inequalities  --  as Hardy  inequality, Heisenberg--Pauli--Weyl uncertainty principle and  Caffarelli--Kohn--Nirenberg inequality -- usually hold on forward complete Finsler manifolds with \textit{non-positive} $S$-curvature, cf.\ Huang, Krist\'aly and Zhao [\textit{Trans.\ Amer.\ Math.\ Soc.}, 2020]. In this paper however we prove that  -- under similar  assumptions    on the reversibility and flag curvature as before -- the aforementioned functional inequalities \textit{fail} whenever the $S$-curvature is \textit{positive}. Accordingly, our results clearly reveal the deep dependence of functional inequalities on the $S$-curvature.
As a  consequence of these results, we establish   analytic aspects of Finsler manifolds, e.g.,  if the flag curvature is non-positive, the Ricci curvature is bounded from below and  the $S$-curvature is positive, then the reversibility turns out to be infinite. Further topological properties and examples are presented on  general Funk metric spaces, where the $S$-curvature plays again a decisive role.


\end{abstract}
\maketitle


\vspace{1cm}
\section{Introduction}
As the title of S. S. Chern's famous 1996 paper \cite{Ch1} shows, his credo  was that `Finsler geometry is just Riemannian geometry without the
quadratic restriction'. In addition, Chern concluded his paper by claiming that  ``(...) \textit{almost all the results
of Riemannian geometry can be developed in the
Finsler setting. It is remarkable that one needs
only some conceptual adjustment, no essential
new ideas being necessary}."
 To some extent, Chern's position  cannot be discredited, since most of the well-known results
from Riemannian geometry -- by suitable adaptations -- have their natural Finslerian counterparts,
e.g.,\ Hopf--Rinow, Cartan--Hadamard, Gauss--Bonnet and Bonnet--Myers theorems as well as Rauch,
Bishop--Gromov and Bishop--G\"unther comparison principles; see  Auslander \cite{Auslander}, Bao--Chern--Shen \cite{BCS}, Rund \cite{Rund}, Shen \cite{ShenLecture} and Ohta \cite{Ohtabook}. However, the development of Finsler geometry over the last three decades clearly shows that fundamentally new ideas are  needed to solve deep problems, revealing in this way genuine differences between Riemannian and Finsler geometries. In support of this thesis, we recall some open questions and spectacular results from Finsler geometry, the solutions to which are far from being tied to the Riemannian world:
\begin{itemize}
	\item \textit{Classification of Finsler manifolds with constant flag curvature.} Unlike the Hopf classification in Riemannian geometry, no full characterization is available
	for Finsler manifolds having constant flag curvature; in fact,  it seems that different subclasses of Finsler  manifolds -- as Minkowski, Berwald, Landsberg, projectively flat and Randers spaces --   play a crucial role in such a characterization, see
	e.g., Bao--Robles--Shen \cite{BRS}, Li \cite{Li}, and
	Shen \cite{Sh0, Shen-Canadian, Sh1}.


	\item {\it Busemann NPC spaces.} Busemann non-positively curved (shortly, NPC) spaces on Riemannian manifolds are characterized by the non-positivity of the sectional curvature, see Busemann \cite{Busemann} and Jost \cite{Jost}. It was shown however by Kelly--Straus \cite{Kelly-St} that any strictly convex domain $D\subset \mathbb R^2$ endowed with the Hilbert metric (which is a projectively flat Finsler manifold with constant negative curvature) is a Busemann NPC space if and only if the boundary of $D$ is an ellipse and the Hilbert geometry is hyperbolic.  Ivanov--Lytchak \cite{Ivanov-Lytchak} proved that  a Finsler manifold is a Busemann NPC space if and only if it is a Berwald manifold of non-positive flag curvature.

	\item {\it Sobolev spaces over Finsler manifolds.} Contrary to the Riemannian setting, the Sobolev spaces over noncompact, irreversible Finsler manifolds should not have
	even a vector space structure, see Krist\'aly--Rudas \cite{KR}. This construction was carried out over the Funk metric defined on the $n$-dimensional Euclidean unit ball $\mathbb B^n$.
\end{itemize}

In addition to the above list, the present paper provides  another new aspect of non-Rieman\-nian Finsler manifolds: we  show that famous functional inequalities -- which naturally hold on Riemannian manifolds --  \textit{fail} on a large class of Finsler manifolds.

Let $(M,F,\m)$ be a  Finsler metric measure manifold ($\FMMM$, for short),
i.e., a Finsler manifold $(M,F)$ endowed with a smooth positive measure $\m$, see Section \ref{prelimain} for details.  Huang--Krist\'aly--Zhao \cite{HKZ} already pointed out that  three non-Riemannian quantities  strongly influence the validity of functional inequalities on  $(M,F,\m)$,  namely,
\begin{itemize}
	\item \textit{reversibility};
	
		\item \textit{flag curvature};
		
		\item $S$-\textit{curvature}, induced by the measure $\m$.

\end{itemize}
The {\it reversibility} of a Finsler manifold $(M,F)$, introduced by Rademacher \cite{Rade}, is defined as
\[
\lambda_F(M):=\sup_{x\in M}\lambda_F(x),\quad \text{ where }\quad \lambda_F(x):=\sup_{y\in T_{x}M\backslash\{0\}}\frac{F(x,-y)}{F(x,y)}.
\]
It is easy to see that $\lambda_F(M)\geq 1$ with equality if and only if $F$ is {\it reversible} (i.e., symmetric). Obviously,
all Riemannian metrics are reversible, but there are infinitely many irreversible Finsler metrics. Typical examples with infinite reversibility are Funk metric spaces,   which served to prove that Sobolev spaces over such Finsler manifolds are not linear structures, see Farkas--Krist\'aly--Varga \cite{FKV} and Krist\'aly--Rudas \cite{KR}. Even if the Finsler metric is explicitly given, it is usually a  hard task to determine whether the reversibility is finite or not.

For each  $y\in  T_xM$ and any plane $  \Pi=\text{span}\{y,v\}\subset T_xM$, the
\emph{flag curvature}  is defined by
\begin{eqnarray}
	\nonumber
	\mathbf{K} (y;\Pi):=\mathbf{K}(y;v):=\frac{g_{y}(R_y(v),v)}{F^2(x,y)\,g_{y}(v,v)-g^2_{y}(v,y)^2},
\end{eqnarray}
where $R_y$ is the Riemannian curvature of $F.$ Clearly, the flag curvature naturally defines the Ricci curvature ${\bf Ric}.$
As expected, on any Riemannian manifold the flag curvature reduces to the sectional curvature.

Unlike in the Riemannian case, there is no canonical measure on  Finsler manifolds.
Hence, various measures can be introduced whose behavior
may be genuinely different. Every smooth positive measure $\m$ on a Finsler manifold $(M,F)$ induces two non-Riemannian quantities, which are the so-called {\it distortion} $\tau$ and {\it $S$-curvature} $\mathbf{S}$, respectively, see \cite{ShenLecture}. More precisely, set $\dm:=\sigma(x){\dd}x^1\wedge\cdots\wedge {\dd}x^n$ in some local coordinate system $(x^i)$. For any $y\in T_xM\backslash\{0\}$, define
\[
\tau(x,y):= \ln \frac{\sqrt{\det g_{ij}(x,y)}}{\sigma(x)}, \qquad \mathbf{S}(x,y):=\left.\frac{\dd}{{\dd}t}\right|_{t=0}\tau(\gamma_y(t), {\gamma}'_y(t)),
\]
where $g_y=(g_{ij}(x,y))$ is the fundamental tensor of $F$ and $t\mapsto \gamma_y(t)$ is the geodesic starting from $x$ with initial speed $\gamma'_y(0)=y$. Since ${\bf S}$ is 1-positive homogeneous, we say that $\mathbf{S}\geq a$ for some constant $a\in \mathbb{R}$ if
$
\mathbf{S}(x,y)\geq a F(x,y)$ for all $ (x,y)\in TM\backslash\{0\}.
$

The distortion and $S$-curvature  are inherently present in Riemannian metric measure manifolds, albeit manifested in different forms.
Indeed, if $(M,g,e^{-f}{\dd}{\vol_g})$ is a weighted Riemannian metric measure manifold, where   $f\in C^\infty(M)$ and ${\dd}{\vol_g}$ is the canonical measure, then a straightforward computation shows that for any $(x,y)\in TM$,
\begin{equation}\label{RiedisS}
\tau(x,y)=f(x),\quad \mathbf{S}(x,y)=g_x(y,\nabla f(x)).
\end{equation}
Clearly, the $S$-curvature does not vanish unless $f$ is a constant.
In the Riemannian setting the $S$-curvature  cannot preserve a constant sign unless it vanishes identically, whereas this can be achieved by Finslerian metrics; for more details, see the example of Funk metric space below.

Before stating our main results,  we recall the main achievements from \cite{HKZ}, which represent the starting points of our present study; in fact, they roughly state that if
\begin{equation}\label{ksnegativ}
	{\bf K}\leq 0 \ \ {\rm and}\ \  {\bf S}\leq 0,
\end{equation}
 then the well-known functional inequalities hold with sharp (Euclidean) constants, including the Hardy  inequality, Heisenberg--Pauli--Weyl uncertainty principle and  Caffarelli--Kohn--Nirenberg inequality, respectively.

 To be more explicit, let $(M,F,\m)$ be an $n$-dimensional  Cartan--Hadamard $\FMMM$ (i.e., simply connected, forward complete $\FMMM$  with ${\bf K}\leq 0$) and ${\bf S}\leq 0$. If $r(x)=d_F(o,x)$ is the distance function from a point $o\in M$, and $F^*$ stands for the co-metric of $F$,   then the following  functional inequalities hold (for their purely Riemannian versions, see Krist\'aly \cite{K2}):
\begin{itemize}
	\item {\it Hardy inequality}; see \cite[Theorem 1.4 for $p=2$]{HKZ} and \cite{ZhaoHardy}. If $p\in (1,n),$ then
\begin{equation}\label{Hardy-1}
	{\ds\int_M \max \left\{ {F}^{*p}({\pm\rm d}u) \right\}\dm}\geq \left(\frac{n-p}{p}\right)^p{\ds\int_M \frac{|u|^p}{r^p}\dm},\ \forall u\in C^\infty_0(M);
\end{equation}
	moreover,  if $F$ is reversible, then the constant $\left(\frac{n-p}{p}\right)^p$
	is sharp and never achieved.
	\item \textit{Heisenberg--Pauli--Weyl  uncertainty principle}; see \cite[Theorem 1.1]{HKZ}. If $n\geq 2,$ then
	\begin{equation}\label{HPW-1}
	\ \ \  \ \ 	{\ds\left(\int_M \max \left\{ {F}^{*2}({\pm\rm d}u) \right\}\dm\right)} \left(\int_M {r^2u^2}\dm\right)\geq \frac{n^2}{4}{\ds\left(\int_M {u^2}\dm\right)^2},\ \forall u\in C^\infty_0(M);
	\end{equation}
	 in addition,  if $F$ is reversible, then  $\frac{n^2}{4}$ is sharp and is achieved by a non-trivial function if and only if  $\m$ is the Busemann--Hausdorff measure $\m_{BH}$, ${\bf K}=0$  and ${\bf S}={\bf S}_{BH}=0$.
	
	\item \textit{Caffarelli--Kohn--Nirenberg interpolation inequality}; see \cite[Theorem 1.1]{HKZ}. If $0<q<2<p$ and $2<n<2(p-q)/(p-2)$, then for every $u\in C^\infty_0(M)$ one has
	\begin{equation}\label{CKN-1}
	{\ds\left(\int_M \max \left\{ {F}^{*2}({\pm\rm d}u) \right\} \dm\right)} \left(\int_M \frac{|u|^{2p-2}}{r^{2q-2}}\dm\right)\geq \frac{(n-q)^2}{p^2}{\ds\left(\int_M \frac{|u|^p}{r^q}\dm\right)^2};
	\end{equation} furthermore,  if $F$ is reversible, then  $\frac{(n-q)^2}{p^2}$ is sharp and is achieved as in the case \eqref{HPW-1}.
\end{itemize}
On the other hand, if $(M,F,\m)$ is an $n$-dimensional reversible  $\FMMM$, and the  ``complementary" case to \eqref{ksnegativ} holds, i.e., $${\bf Ric}\geq 0\ \ {\rm   and}\ \   {\bf S}\geq 0,$$
 together with either the Heisenberg--Pauli--Weyl  uncertainty principle \eqref{HPW-1} or the Caffarelli--Kohn--Nirenberg interpolation inequality \eqref{CKN-1}, then we have the  rigidity   that $\m$ is the Buse\-mann--Hausdorff measure $\m_{BH}$ (up to a multiplicative constant), ${\bf K}=0$  and ${\bf S}={\bf S}_{BH}=0$, see \cite[Theorem 1.3]{HKZ}.

 The above  set of results contains in certain sense the genesis of our first main result  and roughly shows that the \textit{validity/failure} of functional inequalities on Finsler manifolds is decided by the \textit{sign} of the $S$-curvature.

\begin{theorem}\label{mainThm1}
	Let $n\geq 2$ be an integer,  $(M,F,\m)$ be  an $n$-dimensional forward complete $\FMMM$, and $r=d_F(o,\cdot)$ for some $o\in M$. Suppose that
\begin{equation}\label{ric-s-condition}
	\mathbf{Ric}\geq -(n-1) k^2, \quad  \mathbf{S} \geq (n-1){h},
\end{equation}
	where  $k,h$ are two constants with $h> k\geq 0$.
	Then  the following statements hold$:$
	\begin{enumerate}[\rm (i)]

		\item\label{mainThmxx}  the first eigenvalue of the $p$-Laplacian vanishes on $(M,F,\m)$ for every $p>1$, i.e.,
		\[
		\inf_{u\in C^\infty_0(M)\backslash\{0\}}\frac{\ds\int_M {F}^{*p}({\rm d} u) \dm}{\ds\int_M {|u|^p}\dm}=0;
		\]

		\item\label{mainThmxxx} the $L^p$-Hardy inequality  fails on $(M,F,\m)$ for every $p>1$, i.e.,
		\[
		\inf_{u\in C^\infty_0(M)\backslash\{0\}}\frac{\ds\int_M {F}^{*p}({\rm d} u) \dm}{\ds\int_M \frac{|u|^p}{r^p}\dm}=0;
		\]
		
		\item\label{mainThm4} the Heisenberg--Pauli--Weyl principle fails on $(M,F,\m)$, i.e.,
		\[
		\inf_{u\in C^\infty_0({M})\backslash\{0\}}\frac{\left( \ds\int_{{M}} F^{*2}({\rm d} u)\dm \right)\left(  \ds\int_{{M}} r^2u^2 \dm \right) }{\left( \ds\int_{{M}} u^2 \dm  \right)^2}=0;
		\]
		
		\item\label{mainThm5} the Caffarelli--Kohn--Nirenberg interpolation inequality fails on $(M,F,\m)$, i.e., if  $0<q<2<p$ and $2<n<2(p-q)/(p-2)$,
		then
		\[
		\inf_{u\in C^\infty_0({M})\backslash\{0\}}\frac{\left(\ds\int_{{M}} F^{*2}({\rm d} u) \dm   \right)\left(\ds\int_{{M}}  \frac{|u|^{2p-2}}{r^{2q-2}}\dm   \right)  }{\left(\ds\int_{{M}}  \frac{|u|^{p}}{r^{q}}\dm   \right)^2}=0.
		\]
	\end{enumerate}
\end{theorem}

As intuitively described by Shen \cite{Shen-adv-math}, the $S$-curvature  describes the change of measure along geodesics. As far as we know, only a few results are available
in the literature concerning the impact of $S$-curvature on  global properties of  Finsler manifolds. In the light of the results from \cite{HKZ}, Theorem \ref{mainThm1} provides a concluding answer about the validity of functional inequalities in Finsler manifolds, which is captured by the behavior of the $S$-curvature.
  In particular, all the statements  except \eqref{mainThmxx} in
Theorem \ref{mainThm1} will be invalid if the condition $h>0$ is relaxed to $h\geq 0$, showing in this way the optimality of our results.
Indeed, every reversible Minkowski space endowed with the Busemann--Hausdorff measure is a  Cartan--Hadamard manifold with $\mathbf{K}=0$ and $\mathbf{S}=0$ (i.e., $h=k=0$); then the (sharp) inequalities from \eqref{Hardy-1}-\eqref{CKN-1} all are valid, which contradict the conclusions from  Theorem \ref{mainThm1}.
 We also notice that there are large classes of non-Riemannian reversible Minkowski spaces; see Example \ref{reversminkspaces}.


%

The proof of Theorem \ref{mainThm1} proceeds by employing polar coordinates on $M$ with a suitable disintegration of the measure $\m$.
 Combined with the key assumption \eqref{ric-s-condition} with  $h> k\geq 0$, comparison principles and suitably chosen test functions (via appropriate approximation arguments), this setup yields the results. For the proof of (i)-(iii) the Gaussian bubble $u_\alpha=-e^{-\alpha r}$ is used, while for (iv) a Talentian bubble is considered.

 The Gaussian bubble has been used in \cite{HKZ} and Krist\'aly \cite{Kris} to prove (i) and (ii), respectively, in the special case of the Funk metric on the Euclidean unit ball $\mathbb B^n$. More generally, for every non-empty bounded strongly convex domain $\Omega \subset \mathbb{R}^n$ (containing the origin $\mathbf{0}$) with smooth boundary, there exists a Finsler metric $F$, which is a  {\it Funk metric}, and  $(\Omega,F)$ is a Cartan--Hadamard manifold with $\mathbf{K}=-1/4$; see  for example \cite{Li,Sh1,ShenSpray}.
 Such a pair $(\Omega,F)$ is called a {\it Funk metric space} and serves as a model space for the application of Theorem \ref{mainThm1}.  In the specific case $\Omega=\mathbb{B}^n$ and $\m=\m_{BH}$, the  Funk metric  is given explicitly by
\begin{equation}\label{specialfunk}
 F(x,y)
 =\frac{\sqrt{ |y |^2-|x|^2|y|^2+\langle x,y \rangle^2} +\langle x,y\rangle}{1-|x|^2},
 \quad x \in \mathbb{B}^n,\ y \in T_x\mathbb{B}^n=\mathbb{R}^n,
\end{equation}
 as in Shen \cite{Sh1,ShenSpray,ShenLecture}. For this metric, one has ${\bf Ric}=-(n-1)/4$ and ${\bf S}_{BH}=(n+1)/2$,  and hence $h=\frac{n+1}{2(n-1)}>\frac{1}{2}=k>0$.
By contrast, the symmetrized metric
$
\tilde{F} (x,y)= ( F(x,y) + F(x, -y))/2,
$
is precisely the Riemannian Hilbert/Klein metric on $\mathbb{B}^n$ with ${\bf K}=-1$ (so ${\bf Ric}=-(n-1)$) and ${\bf S}=0$.
  In this reversible setting, the sharp inequalities \eqref{Hardy-1}--\eqref{CKN-1} hold, and Theorem~\ref{mainThm1} does not apply.

In the sequel, we investigate the influence of $S$-curvature on  analytic properties of Finsler manifolds; more precisely, it turns out that the positivity of $S$-curvature has a strong implication on the
reversibility.

\begin{theorem}\label{reversibinfite}
Let $(M,F,\m)$ be  an $n$-dimensional forward complete noncompact $\FMMM$. Assume that
\begin{equation}\label{second-cond}
\mathbf{K}\leq 0,\quad \mathbf{Ric}\geq -(n-1) k^2, \quad  \mathbf{S} \geq (n-1){h},	
\end{equation}
where  $k,h$ are two constants satisfying $h>0$ and $h\geq k\geq 0$. Then $\lambda_F(M)=+\infty$.
\end{theorem}
Obviously, every Funk metric space endowed with the Busemann--Hausdorff measure verifies the assumptions of
Theorem \ref{reversibinfite}. The additional condition $\mathbf{K}\leq 0$ is mild. In fact, if the flag curvature is bounded from below by a positive number, then the Bonnet--Myers theorem  implies the compactness of the manifold, see  \cite[Theorem 7.7.1]{BCS}, in which case the reversibility is finite. In addition, the assumptions on both the flag and Ricci curvatures can be removed in the case of projectively flat Finsler manifolds, see Theorem \ref{Spositivenoncom} below.

Before stating the next result, we recall that another important notion in Finsler geometry is the \textit{weighted Ricci curvature} $\mathbf{Ric}_N$ for $N\in [n,\infty]$,  whose lower bound $\mathbf{Ric}_N\geq K$ for some $K\in \mathbb R$ is equivalent to the validity of the  CD$(K,N)$ condition in the Finsler setting (see Ohta \cite{Ohtabook,OS,Ohta-pacific}).
 Although $\mathbf{Ric}_N$ incorporates both $\mathbf{Ric}$ and $\mathbf{S}$,
 it cannot replace the separate conditions in \eqref{ric-s-condition} or \eqref{second-cond}.
  Indeed, any reversible Minkowski space with the Busemann--Hausdorff measure satisfies
$\mathbf{Ric}_N=0$; however, on such a space,  neither Theorem \ref{mainThm1} nor Theorem \ref{reversibinfite} holds (see Subsection \ref{Counterexamples}).
Furthermore,
assuming $\mathbf{Ric}_N\geq K>0$  is also unsuitable, as this implies the manifold is compact and hence has finite reversibility.

The $S$-curvature also  governs the topology of projectively flat Finsler manifolds.
A Finsler metric $F$ on  a domain $\Omega\subset \mathbb{R}^n$ is   {\it projectively flat} if the images of geodesics are straight lines.
The characterization of such metrics on domains in $\mathbb{R}^n$  constitutes the regular case of Hilbert's fourth problem.
Although Beltrami's theorem ensures that every projectively flat Riemannian metric is of constant sectional curvature and vice versa, this is not true for non-Riemannian Finsler metrics, cf. \cite{ChengLi, Li, Sh1}; this fact motivates the study of such Finsler manifolds.
The following theorem reveals that
 a positive lower bound of $S$-curvature influences not only the
 boundedness and completeness of the domain but also the reversibility of a projectively flat Finsler manifold.
\begin{theorem}\label{Spositivenoncom}
Let  $(\Omega,F,\mathscr{L}^n)$ be an $n$-dimensional forward complete   projectively flat Finsler manifold  endowed with the Lebesgue measure.
Assume that    $\mathbf{S} \geq (n-1)h $ for some constant $h>0$.
Then the following statements hold$:$
\begin{enumerate}[\rm (i)]
\item  $\Omega$ is a bounded domain  in $\mathbb{R}^n;$

\item  $(\Omega,F)$ is  not backward complete$;$

\item $\lambda_F(\Omega)=+\infty$.

\end{enumerate}
\end{theorem}
Note that some statements remain valid under a weaker condition $\mathbf{S}>0$; see subsection \ref{projceflatmanifold} for details. We also remark that by using $\mathbf{S}\leq -(n-1)h$ instead of $\mathbf{S}\geq (n-1)h$,
Theorems \ref{mainThm1}--\ref{Spositivenoncom} can be established
 for backward
complete Finsler manifolds.

The paper is organized as follows. Section \ref{prelimain} is devoted to preliminaries on Finsler
geometry together with some results about volume comparison. In Section \ref{Sobolevsapce}, we provide some immediate properties of Sobolev spaces over irreversible Finsler manifolds. Here, we also establish an approximation result which will be useful in our main results.  The proof of our main Theorem \ref{mainThm1} is given in Section \ref{influeofS1}, while the proofs of Theorems \ref{reversibinfite} and \ref{Spositivenoncom} are presented in Section \ref{infactlue2}.  Section \ref{sectionFunk} is devoted to the study of Funk metric spaces, by providing examples and counterexamples supporting the sharpness of our main results.

\section{Preliminaries}\label{prelimain}
\subsection{Elements from Finsler geometry} In this section we recall some definitions
and properties from Finsler geometry; for details see e.g.\ Bao--Chern--Shen
\cite{BCS} and Shen \cite{ShenSpray,ShenLecture}.
\subsubsection{Finsler manifolds}
 A {\it Minkowski norm}  on an $n$($\geq 2$)-dimensional vector space $V$ is a $C^{\infty}$-function  $\phi : V \setminus \{ 0 \} \rightarrow [0, +\infty)$ satisfying:
\begin{enumerate}[{\rm (i)}]
\item $\phi (y) \geq 0$ with equality if and only if $y=0$;
\item $\phi$ is positively homogeneous of degree one, i.e., $\phi(\alpha y) = \alpha \phi(y)$ for all $\alpha > 0$;
\item $\phi(y)$ is strongly convex, i.e.,  the matrix $g_{ij}(y) :=   \frac{1}{2} \frac{\partial^2\phi^2}{\partial y^i\partial y^j}  (y)$ is positive definite;
\end{enumerate}
having these properties, the pair $(V,\phi)$ is called a {\it Minkowski space}.

Let $M$ be an $n$-dimensional (connected) smooth manifold and $TM = \bigcup_{x\in M} T_xM$ be its tangent bundle.
 A {\it Finsler metric} $F = F(x,y)$ on  $M$ is a $C^{\infty}$ function on $TM \setminus \{ 0 \}$ such
 that $F|_{T_{x}M}$ is a Minkowski  norm on $T_{x}M$ for each $x\in M$, in which case the pair ($M$, $F$) is called a {\it Finsler manifold}.

The {\it fundamental tensor} is defined by
\begin{align}\label{defbasictensor}
g_y := \big(g_{ij}(x,y)\big) = \left(\frac{1}{2}\frac{\partial^2 F^2}{\partial y^i \partial y^j}(x,y)\right),\quad \forall\,(x,y)\in TM\backslash\{0\},
\end{align}
which induces a Riemannian metric on $T_xM\backslash\{0\}$.
The Finsler metric $F$ is (induced by)   a    Riemannian metric if and only if $g_y$ is independent of $y$, i.e., $g_{ij}(x,y) = g_{ij}(x)$.
In the same way, $F$ is a {\it   Minkowski metric} if   $g_y$ is independent of $x$, i.e., $g_{ij}(x,y) = g_{ij}(y)$.

The {\it reversibility} of a set $U\subset M$, following Rademacher \cite{Rade},  is defined by
\begin{equation}\label{def_reversibility}
\lambda_F(U) := \sup_{x \in U} \lambda_{F}(x), \quad \text{where} \quad \lambda_{F}(x) = \sup_{y \in T_xM \setminus \{ 0 \}} \frac{F(x, -y)}{F(x,y)}.
\end{equation}
Observe that $\lambda_F(M) \geq 1,$ with equality if and only if $F$  is {\it reversible}, i.e., $F (x, y) = F(x, -y)$ for all $(x,y)\in TM$. Moreover, $x\mapsto \lambda_F(x)$ is continuous on $M$, and hence
 $\lambda_F(M) $ is finite whenever $M$ is compact or $F$ is a Minkowski metric. 

The {\it Legendre transformation} $\mathfrak{L} : TM \rightarrow T^*M$
  is defined by
\[
\mathfrak{L}(x,y):=
\begin{cases}
g_y(y,\cdot), &\text{ if }y\in T_xM\backslash\{0\},\\
0, &\text{ if }y= 0\in T_xM,
\end{cases}
\]
where $g_y(y,\cdot):=g_{ij}(x,y)y^i{\dd}x^j$. Given a smooth function $u: M \rightarrow \mathbb{R}$,
its {\it gradient} $\nabla u(x)$ is defined as
\[
\nabla u(x) = \mathfrak{L}^{-1} ( {\dd} u).
\]
 This leads to the  identity $\langle X, {\dd}u\rangle=X(u)=g_{\nabla u}(\nabla u,X)$ valid for any smooth vector field $X$. Here $\langle y,\xi\rangle:=\xi(y)$
 denotes the canonical pairing between $T_xM$
and $T^*_xM$.


The  {\it co-metric} $F^{*}$ of $F$ on $M$ is defined by
\begin{equation}\label{dualmetric}
F^{*}(x,\xi):=\sup_{y \in T_xM \setminus \{0\}} \frac{\xi(y)}{F(x,y)}, \quad \forall \xi \in T^{*}_xM,
\end{equation}
which is a Finsler metric on the cotangent bundle $T^{*}M$. It is easy to see that
\begin{equation}\label{dualff*}
\langle y,\xi\rangle\leq F(x,y)F^*(x,\xi), \quad  \text{ for any }y\in T_xM,\ \xi\in T^*_xM,
\end{equation}
with equality if and only if $\xi=\alpha \mathfrak{L}(x,y)$ for some $\alpha>0$. In addition, $F(x,y)=F^*(\mathfrak{L}(x,y))$ for any $y\in T_xM$. Furthermore, the reversibilities of $F$ and $F^*$ coincide, see Huang--Krist\'aly--Zhao \cite[Lemma 2.1]{HKZ},  i.e.,
\begin{equation}\label{revercondin}
\lambda_F(x)=\lambda_{F^*}(x), \quad \forall\,x\in M.
\end{equation}

A smooth curve $t\mapsto\gamma(t)$ in $(M, F)$ is called a {\it geodesic} if it satisfies
\begin{equation}\label{geodesequ}
\frac{{\dd}^2 \gamma^i}{{\dd}t^2} + 2 G^i \left(\gamma, \frac{{\dd} \gamma}{{\dd}t}\right) = 0,
\end{equation}
where $ G^i = G^i(x,y)$ are the {\it geodesic coefficients }(or  {\it spray coefficients})
given by
\begin{equation}\label{goedcooff}
G^i = \frac{1}{4} g^{ij} \left\{ \frac{\partial^2 F^2}{\partial x^l  \partial y^j} y^l -\frac{\partial F^2}{\partial x^j} \right\},
\end{equation}
 $(g^{ij})$ being the inverse matrix of $(g_{ij})$.
In the sequel, we always use $\gamma_y(t)$, $t\geq 0,$ to denote the geodesic with initial velocity $y$.

 For each  $y\in  T_xM$ and any plane $  \Pi=\text{span}\{y,v\}\subset T_xM$, the
\emph{flag curvature}  is defined by
\begin{eqnarray}
\nonumber
\mathbf{K} (y;\Pi):=\mathbf{K}(y;v):=\frac{g_{im}(x,y)\,R^i_{\ k}(x,y)\, v^kv^m}{F^2(x,y)\,g_{ij}(x,y)\,v^iv^j-[g_{ij}(x,y)\,y^iv^j]^2},
\end{eqnarray}
where
\begin{eqnarray}
\nonumber R^i_{\ k}=2\frac{\partial G^i}{\partial
x^k}-y^j\frac{\partial^2G^i}{\partial x^j\partial
y^k}+2G^j\frac{\partial^2G^i}{\partial y^j\partial
y^k}-\frac{\partial G^i}{\partial y^j}\frac{\partial G^j}{\partial
y^k}.
\end{eqnarray}
If $F$ is a  Riemannian metric, the flag curvature reduces
 to the sectional curvature. 

The {\it Ricci curvature} of $y\in T_xM\backslash\{0\}$ is defined by
\[
\mathbf{Ric}(y):=\frac{1}{F^2(x,y)}\sum_{i}\mathbf{K}(y;e_i),
\]
where $\{e_1,\ldots,e_n\}$ is a $g_y$-orthonormal basis for $T_xM$.

Given a locally Lipschitz curve $c : [0, 1] \rightarrow M$, its {\it length} is defined by
 \[
 L_{F}(c) = \int_0^1 F(c(t), c'(t)) {\dd}t.
 \]
The {\it distance function} $d_{F} : M \times M \rightarrow [0, \infty )$  is defined  as
$d_{F}(x_1, x_2):= \inf L_F(c),$
where the infimum is taken over all locally Lipschitz curves
$c:[0,1] \rightarrow M$ with $c(0) =x_1$ and $c(1)=x_2$.
For any $x_1,x_2,x_3\in M$, it satisfies $d_F(x_1,x_2)\geq 0,$ with equality if and only if $x_1=x_2$, and $d_F(x_1,x_2)\leq d_F(x_1,x_3)+d_F(x_3,x_2)$.
Usually $d_F(x_1, x_2) \neq d_F(x_2, x_1),$ unless $F$ is reversible. In fact, one has for every  $x_1,x_2\in M$  that
\begin{equation}\label{non-symmetric}
	\frac{d_F (x_1, x_2)}{d_F(x_2, x_1)} \leq \lambda_F(M).
\end{equation}
Given $R > 0$, the {\it forward} and {\it backward metric balls} $B^+_x(R)$ and $B^-_x(R)$ are defined by
 \[
 B^+_x(R):=\{ z\in M\,:\, d_F(x,z)<R\},\quad B^-_x(R):=\{ z\in M\,:\, d_F(z,x)<R\}.
 \]
If $F$ is reversible, forward and backward metric balls coincide which are denoted
by $B_x(R)$.

A Finsler manifold $(M, F)$ is called {\it forward complete} if every geodesic $t\mapsto \gamma(t)$, $t\in [0,1)$, can be
extended to a geodesic defined on $t\in [0,+\infty)$; similarly, $(M, F)$ is called {\it backward complete} if every geodesic $t\mapsto \gamma(t)$, $t\in (0,1]$, can be extended to a geodesic on $ t\in (-\infty,1]$. We say that $(M, F)$ is   {\it complete} if it is both forward
 and backward complete. According to \cite{BCS}, if $(M,F)$ is either forward or backward complete, then for every two points $x_1,x_2\in M$, there exists a minimal geodesic $\gamma$ from $x_1$ to $x_2$ such that $d_F(x_1,x_2)=L_F(\gamma)$. Moreover, the closure of a forward (resp., backward) metric ball with finite radius is compact if $(M,F)$ is forward (resp., backward) complete.

The {\it reverse Finsler metric} $\overleftarrow{F}$ is defined as $\overleftarrow{F}(x,y):=F(x,-y)$. It is easy to check that $(M,F)$ is forward (resp., backward) complete   if and only if $(M,\overleftarrow{F})$ is backward (resp., forward) complete. For this reason, we focus in the sequel only on  forward complete Finsler manifolds.

\subsection{Measure and polar coordinate system} In this subsection, let $(M,F,\m)$ be a forward complete {\it Finsler metric measure manifold} ($\FMMM$, for short), i.e., a forward complete Finsler manifold $(M,F)$ endowed with a smooth positive measure $\m$.
In a local coordinate system ($x^i$), $\dm$ can be expressed as
\begin{equation}\label{measure_m}
\dm = \sigma {\dd}x^1 \wedge \cdot\cdot\cdot \wedge {\dd}x^n,
\end{equation}
where $\sigma= \sigma(x)$ denotes the {\it density function} of $\dm$. In particular, the {\it Busemann--Hausdorff measure}
$\dm_{BH}$ is defined by
\[
\dm_{BH}:=\frac{\vol(\mathbb{B}^n)}{\vol(B_xM)}{\dd}x^1 \wedge \cdot\cdot\cdot \wedge {\dd}x^n=\frac{\omega_n}{\vol(B_xM)}{\dd}x^1 \wedge \cdot\cdot\cdot \wedge {\dd}x^n,
\]
where $\mathbb{B}^n$ is the  $n$-dimensional  Euclidean
unit ball, $\omega_n:=\vol(\mathbb{B}^n)$ and $B_xM:=\{y\in T_xM\,:\, F(x,y)<1\}$.

The {\it distortion} $\tau$ and the {\it S-curvature} $\mathbf{S}$ of $(M,F,\m)$ are defined  by
\begin{equation}\label{distsdef}
\tau(x,y):= \ln \frac{\sqrt{\det g_{ij}(x,y)}}{\sigma(x)}, \qquad \mathbf{S}(x,y):=\left.\frac{\dd}{{\dd}t}\right|_{t=0}\tau(\gamma_y(t), {\gamma}'_y(t)), \quad \forall\, y\in T_xM \setminus \{0\},
\end{equation}
respectively, where  $t\mapsto \gamma_y(t)$ is a geodesic with $\gamma'_y(0)=y$. According to Shen \cite{ShenLecture}, we also have the explicit form
\begin{equation}\label{Scurvature}
\mathbf{S}(x,y) = \frac{\partial G^m}{\partial y^m}(x,y) - y^k \frac{\partial}{\partial x^k}(\ln \sigma(x)),
\end{equation}
where $G^m$ are the geodesic coefficients.
Throughout this paper, we say that $\mathbf{S}\geq a$ for some constant $a\in \mathbb{R}$ if
\[
\mathbf{S}(x,y)\geq a F(x,y),\quad \forall\,(x,y)\in TM\backslash\{0\}.
\]


Given a point $o\in M$, set $S_oM:=\{ y \in T_oM: F(o,y) =1 \}$. The {\it cut value} $i_y$ of a direction $y \in S_o M$  and the {\it injectivity } $\mathfrak{i}_o$ at $o$ are defined by
\begin{align*}
	i_y := \sup \{ t>0: \text{the geodesic }  \gamma_y|_{[0,t]}   \text{ is globally minimizing} \}, \quad
	\mathfrak{i}_o  := \inf_{y \in S_oM} i_y>0.
\end{align*}
Moreover, if $(M,F)$ is a Cartan--Hadamard manifold (i.e., a simply connected,  forward complete Finsler manifold with $\mathbf{K}\leq 0$), then $\mathfrak{i}_x=+\infty$ for every $x\in M$.

 Given a point $o\in M$, one can define the {\it polar coordinate system} $(r,y)$ around $o$, where $ r(x):=d_F(o,x)$ and $y\in S_oM$;
 see Zhao--Shen \cite[Section 3]{ZS} for example. In particular,  for any $y\in S_oM$ and $r\in (0,i_y)$,
\begin{equation}\label{geommeaingofr}
\gamma_y(r)=(r,y),\quad \gamma'_y(r)=\nabla r|_{(r,y)}.
\end{equation}
Moreover,
it follows by \cite[Lemma 3.2.3]{ShenLecture} that the eikonal equations hold, i.e.,  $F(\nabla r) = F^{*}({\dd}r) =1$ for  $\m$-a.e. in $M$.

In the polar coordinate system $(r,y)$, the measure $\m$  decomposes as
\begin{equation}\label{volumeexprsso}
\dm|_{(r,y)}=:\hat{\sigma}_o(r,y){\dd}r \wedge {\dd}\nu_o(y),
\end{equation}
where ${\dd}\nu_o(y)$ is the Riemannian volume form of $S_oM$ induced by $g_y$. Furthermore, it follows from  \cite[Lemma 3.1]{ZS} that
\begin{equation}\label{volumerzero}
\lim_{r\rightarrow 0^+}\frac{\hat{\sigma}_o(r,y)}{r^{n-1}}=e^{-\tau(y)}.
\end{equation}


The following notations are useful in the sequel: $\pi/\sqrt{k}:=+\infty$ if $k\leq 0$,
\[
 \mathfrak{s}_k(t):=\left\{
	\begin{array}{lll}
		\frac{\sin(\sqrt{k}t)}{\sqrt{k}}, && \text{ if }k>0,\\
		\ \ \ \ t, && \text{ if }k=0,\\
		\frac{\sinh(\sqrt{-k }t)}{\sqrt{-k}}, && \text{ if }k<0,
	\end{array}
	\right. \qquad
	\mathfrak{c}_k(t):=\frac{\dd}{{\dd}t}\mathfrak{s}_k(t)=\left\{
	\begin{array}{lll}
		{\cos(\sqrt{k}t)}, && \text{ if }k>0,\\
		\ \ \ \ \  1, && \text{ if }k=0,\\
		{\cosh(\sqrt{-k}t)}, && \text{ if }k<0.
	\end{array}
	\right.
\]
In view of \eqref{volumeexprsso}, we recall a volume comparison result, see Zhao--Shen \cite[Theorem 3.4]{ZS}.
\begin{theorem} {\rm (Zhao--Shen \cite{ZS})}\label{bascivolurcompar} Let $(M,F,\m)$ be an $n$-dimensional forward complete $\FMMM$ with $\mathbf{Ric}\geq (n-1) k$. For every $o\in M$, there holds
\begin{equation}\label{densesim}
\hat{\sigma}_o(r,y)\leq e^{-\tau({\gamma}_y'(r))}  \mathfrak{s}_k^{n-1}(r),\quad r\in (0,i_y), \quad \forall\,y\in S_oM,
\end{equation}
where $(r,y)$ is the polar coordinate system at $o$.
Moreover,  for any $R\in (0, \pi/\sqrt{k})$,
\[
\m[B^+_o(R)]\leq   \int_{S_oM} \left(\int^R_0 e^{-\tau({\gamma}_y'(t))}  \mathfrak{s}_k^{n-1}(r){\dd}r  \right) {\dd}\nu_o(y),
\]
 with equality for some $R_0\in(0, \pi/\sqrt{k})$ if and only if for any $y\in S_oM$,
\[
\mathbf{K}(\gamma_y'(r);\cdot)\equiv k\quad \text{ for }0\leq r\leq R_0\leq  \mathfrak{i}_o,
\]
in which case
\[
\hat{\sigma}_o(r,y)=e^{-\tau({\gamma}_y'(r))}  \mathfrak{s}_k^{n-1}(r),\quad r\in (0,R_0).
\]
\end{theorem}
\begin{remark}\label{impvoluremark} Due to \eqref{geommeaingofr}, for $\m$-a.e.\ $x\in M$, the gradient $\nabla r|_x$ is the velocity of the minimal geodesic from $ o$ to $x$ at the point $x$.
By this observation and the same proof of  \cite[Theorem 3.4]{ZS}, it is not hard to check that the statements in Theorem \ref{bascivolurcompar} remain valid if the curvature condition is weakened to
$
\mathbf{Ric}(\nabla r)\geq (n-1) k.
$
Moreover, if we additionally assume $\mathbf{S}(\nabla r)\geq (n-1)h$, then it follows by \eqref{geommeaingofr} that for any $y\in S_oM$,
\begin{align*}
\tau({\gamma}_y'(r))-\tau(y)=\tau({\gamma}_y'(r))-\tau({\gamma}_y'(0))=\int^r_0 \mathbf{S}({\gamma}_y'(s)){\dd}s\geq (n-1)hr,
\end{align*}
which combined with  \eqref{densesim} yields the key estimate in our forthcoming arguments, i.e.,
\begin{equation}\label{st11ong}
\hat{\sigma}_o(r,y)\leq e^{-\tau({\gamma}_y'(r))}  \mathfrak{s}_{k}^{n-1}(r)\leq e^{-\tau(y)-(n-1)hr}\mathfrak{s}_{k}^{n-1}(r),\quad r\in (0,i_y).
\end{equation}
\end{remark}

As in \cite{HKZ}, we  recall the {\it integral of distortion}
\begin{equation}\label{cocont}
\mathscr{I}_{\m}(x):=\int_{S_xM} e^{-\tau(y)}{\dd}\nu_x(y)<+\infty.
\end{equation}
This quantity expresses a close relation to the measure; in fact, $\mathscr{I}_{\m}$ is constant if and  only if $\m=C \m_{BH}$ for some constant $C>0$. Moreover,
\begin{equation}\label{budist}
\mathscr{I}_{\m_{BH}}(x)=\vol(\mathbb{S}^{n-1})=n\omega_n,\quad \forall\,x\in M,
\end{equation}
 where $\mathbb{S}^{n-1}$ is the   $(n-1)$-dimensional  Euclidean unit sphere.

Throughout this paper,
  $\chi_U$  denotes the  {\it characteristic function} of a set $U\subset M$. As an application of Theorem \ref{bascivolurcompar} and Remark \ref{impvoluremark}, we obtain the following result.

\begin{lemma}\label{dulemma2}
Let $(M,F,\m)$ be an $n$-dimensional forward  complete $\FMMM$ and let $o$ be a fixed point  in $M$. Set $r(x):=d_F(o,x)$ and $ \mu:=  r^{-p} \m$ for some $p\in \mathbb{R}$. Then
\begin{enumerate}[{\rm (i)}]
\item\label{firtesk1}  if $p\in (-\infty,n)$, then  $\ds\int_{M} \chi_\mathscr{K} {\dd} \mu   < + \infty $ for any compact set $\mathscr{K}\subset M$;

\item\label{firtesk2} if $p\in [n, \infty)$, then $\ds\int_{M} \chi_{B^+_o(R)}  {\dd} \mu =   + \infty$ for any $R>0$.

\end{enumerate}
Hence, $\mu=  r^{-p} \m$ is a locally finite Borel measure if and only if  $p\in (-\infty,n)$.
\end{lemma}
\begin{proof}\eqref{firtesk1}
Given a  compact set $\mathscr{K}\subset M$,  since $r(x)$ is a continuous function, there is  some  $R\in (0,+\infty)$ such that  $\mathscr{K}\subset  {B^+_o(R)}$. Because the forward completeness implies the compactness of $ \overline{B^+_o(R)}$, we may assume that
\[
\mathbf{Ric}|_{\overline{B^+_o(R)}}\geq -(n-1)k^2, \quad \mathbf{S}|_{\overline{B^+_o(R)}} \geq (n-1) h,
\]
for some constants $k>0$ and $h\in \mathbb{R}$. It follows by Remark \ref{impvoluremark} that
 for any $y\in S_oM$,
\begin{equation}\label{st11ong2}
\hat{\sigma}_o(r,y)\leq e^{-\tau({\gamma}_y'(r))}  \mathfrak{s}_{-k^2}^{n-1}(r)=e^{-\tau(y)-(n-1)hr}\mathfrak{s}_{-k^2}^{n-1}(r),\quad r\in (0,\min\{R,i_y\}).
\end{equation}
On the other hand,  due to (\ref{volumerzero}),
there exists a small $\varepsilon_0\in (0,\min\{\mathfrak{i}_o,R\})$ such that
\begin{equation}\label{st1short}
\hat{\sigma}_o(r,y) \leq 2 e^{-\tau(y)}r^{n-1}, \quad \forall\,  r \in (0, \varepsilon_0), \ \forall\,y\in S_oM.
\end{equation}

By \eqref{volumeexprsso}, \eqref{cocont}, \eqref{st11ong2} and \eqref{st1short}, we obtain
\begin{align*}
\int_{M} \chi_\mathscr{K} {\dd} \mu&\leq \int_{M} \chi_{B^+_o(R)} {\dd} \mu\leq \int_{B^+_{o}(\varepsilon_0)} r^{-p}\dm + \int_{B^+_o(R) \setminus B^+_{o}(\varepsilon_0)}  r^{-p} \dm\\
&\leq  2\mathscr{I}_{\m}(o)\left( \int^{\varepsilon_0}_0   r^{n-1 - p} {\dd}r \right)+\mathscr{I}_{\m}(o)\left( \int^{R}_{\varepsilon_0}  e^{-(n-1)hr} \mathfrak{s}_{-k^2}^{n-1}(r) r^{-p} {\dd}r \right) \\
& \leq  \frac{2 \mathscr{I}_{\m}(o)}{n - p} \varepsilon_0^{n-p}  + \mathscr{I}_{\m}(o) \frac{\varepsilon_0^{-p}}{(2k)^{n-1}} \int^{R}_{\varepsilon_0} e^{(n-1)(k-h)r} {\dd}r<+\infty.
\end{align*}

\smallskip

\eqref{firtesk2}  For any $R>0$,
the limit (\ref{volumerzero}) furnishes a small $\delta_0\in (0, \min\{\mathfrak{i}_o,R\})$ such that
\[
\hat{\sigma}_o(r,y) \geq \frac{1}{2} e^{-\tau(y)}r^{n-1}, \quad \forall\,  r \in (0, \delta_0), \ \forall\,y\in S_oM,
\]
which combined with \eqref{volumeexprsso} and \eqref{cocont} provides the estimate
\begin{align*}
\ds\int_{M} \chi_{B^+_o(R)}  {\dd} \mu &  \geq
 \ds\int_{B^+_o(\delta_0)} r^{-p} \hat{\sigma}_o(r,y) {\dd}r {\dd}\nu_o(y) \geq \frac{\mathscr{I}_{\m}(o)}{2}  \int_{0}^{\delta_0} r^{n-1 -p} {\dd}r =+\infty,
\end{align*}
where we used the assumption $p\in [n,+\infty)$.
\end{proof}


\section{Sobolev spaces over Finsler manifolds and approximation}\label{Sobolevsapce}

\begin{definition}\label{soboloevespace}Let $(M,F,\m)$ be  an $n$-dimensional forward complete $\FMMM$.
Given $u\in C^\infty_0(M)$ and $p\in [1,\infty)$, define a pseudo-norm as
\begin{equation}\label{W1p}
\|u\|_{W^{1,p}_{\m}}:=\|u\|_{L^p_{\m}}+\|{\dd} u\|_{L^p_{\m}}:=\left(\int_M |u|^p \dm \right)^{1/p}+\left(\int_M F^{*p}({\dd} u)  \dm  \right)^{1/p}.
\end{equation}
The  {\it  Sobolev space} $W^{1,p}_0(M,\m)$ is defined as the closure of $C^\infty_0(M)$ with respect to the backward topology induced by $\|\cdot\|_{W^{1,p}_{\m}}$, i.e.,
\[
u\in W^{1,p}_0(M,\m) \quad \Longleftrightarrow \quad \exists \, (u_n)\subset C^\infty_0(M) \text{ satisfying }\lim_{n\rightarrow \infty}\|u-u_n\|_{W^{1,p}_{\m}}= 0.
\]
\end{definition}

\begin{remark}\label{sobonormprop}
By the properties of $F^*$, it is not hard to verify that for any $u,v\in  W^{1,p}_0(M,\m)$,
 \begin{itemize}
\item $\|u\|_{W^{1,p}_{\m}}\geq 0$ with equality if and only if $u=0$;

\item $\|\alpha\,u\|_{W^{1,p}_{\m}}= \alpha \|u\|_{W^{1,p}_{\m}}$ for any $\alpha>0$;

\item $\|u+v\|_{W^{1,p}_{\m}}\leq \|u\|_{W^{1,p}_{\m}}+\|v\|_{W^{1,p}_{\m}}$.

\end{itemize}
\end{remark}

\begin{proposition}\label{sobolevspaceline}Let $(M,F,\m)$ be an $n$-dimensional forward complete $\FMMM$.
If   $W^{1,p}_0(M,\m)$ is not a vector space for some $p\in [1,\infty)$, then $\lambda_F(M)=+\infty$.
 \end{proposition}
\begin{proof} Due to Remark \ref{sobonormprop}, we clearly have that $\alpha_1u_1+\alpha_2u_2\in W^{1,p}_0(M,\m)$ for every $u_1,u_2\in W^{1,p}_0(M,\m)$ and $\alpha_1,\alpha_2\in \mathbb{R}_+$.

Since $W^{1,p}_0(M,\m)$ is not a vector space for some $p\in [1,\infty)$, there exists $u\in W^{1,p}_0(M,\m)$ with $\|-u\|_{W^{1,p}_{\m}}=+\infty$, i.e.,
$\int_M F^{*p}(-{\dd}u) \dm=+\infty$.
Suppose by contradiction that $\lambda_F(M)<+\infty$. Then  \eqref{revercondin} yields $\lambda_F(M)=\lambda_{F^*}(M)$ and hence,
\[
\int_M F^{*p}(-{\dd}u) \dm\leq \lambda^p_F(M)\int_M F^{*p}({\dd}u) \dm<+\infty,
\]
which is a contradiction.
\end{proof}

\begin{proposition}
Let $(M,F,\m)$ be an $n$-dimensional forward complete $\FMMM$. Suppose that there exist $p\in [1,\infty)$ and $q\in [1,p)$ satisfying
\begin{equation}\label{basiccondition}
\int_M \lambda^{\frac{qp}{p-q}}_F(x)\dm(x)<+\infty.
\end{equation}
Then  $\spa \{u\in W^{1,p}_0(M,\m)\}$ is a linear subspace of $W^{1,q}_0(M,\m)$.
\end{proposition}

\begin{proof}
	 Since $\lambda_F(x)\geq 1$ for every $x\in M$, we have
$
\m(M)\leq \int_M \lambda^{\frac{qp}{p-q}}_F(x)\dm(x)<+\infty.
$
Hence, without loss of generality, we may assume $\m(M)=1$ for convenience.

Given any $u\in W^{1,p}_0(M,\m)$, there exists a sequence $(u_k)\subset C^\infty_0(M)$ such that $\|u-u_k\|_{W^{1,p}_{\m}}\rightarrow 0$.
 H\"older's inequality then yields
$\|u-u_k\|_{W^{1,q}_{\m}}\leq \|u-u_k\|_{W^{1,p}_{\m}}\rightarrow 0$, i.e., $u\in W^{1,q}_0(M,\m)$.
Moreover, \eqref{basiccondition} combined with H\"older's inequality again gives
\begin{align*}
\|{\dd}(-u)-{\dd}(-u_k)\|_{L^q_{\m}}&=\left(\int_M F^{*q}({\dd}u_k-{\dd}u)\dm \right)^{1/q}\leq \left(\int_M \lambda_F^q F^{*q}({\dd}u-{\dd}u_k)\dm \right)^{1/q}\\
&\leq  \left( \int_M \lambda_F^{\frac{qp}{p-q}}\dm   \right)^{\frac{p-q}{qp}} \left(  \int_M F^{*p}({\dd}u-{\dd}u_k)\dm \right)^{\frac{1}{p}}\leq \|\lambda_F\|_{L^{\frac{qp}{p-q}}_{\m}} \| u- u_k\|_{W^{1,p}_{\m}}\\&\rightarrow 0,
\end{align*}
which implies $-u\in W^{1,q}_0(M,\m)$.
\end{proof}

In the rest of this section we consider the approximation by Lipschitz functions, which plays an important role in the proof of our main results.

\begin{definition}
A function $f:M\rightarrow \mathbb{R}$ is called {\it Lipschitz} if there exists a constant $C\geq 0$ such that
\[
|f(x_1)-f(x_2)|\leq  C\,d_F(x_1,x_2),\quad \forall\,x_1,x_2\in M.
\]
The function $f$ is also called {\it $C$-Lipschitz}.
The minimal $C$ satisfying the above inequality is called the {\it Lipschitz constant} or {\it dilatation} of $f$, denoted by $\dil(f)$. 
\end{definition}


For convenience, in the sequel  we use $\Lip(M)$   to denote the collection of Lipschitz functions on $M$. Besides, let $\Lip_0(M):=C_0(M)\cap \Lip(M)$ be the collection of Lipschitz functions  with compact support.


\begin{lemma}\label{localtogloaLip} Let $(M,F)$ be a forward complete Finsler manifold. Then $C^\infty_0(M)\subset \Lip_0(M)$.
\end{lemma}
\begin{proof}It suffices to show $C^\infty_0(M)\subset \Lip(M)$. In fact, for every $u\in C^\infty_0(M)$, given every $x_1,x_2\in U:=\supp u$, the forward completeness of $(M,F)$ yields a unit speed minimal geodesic $\gamma(t)$, $t\in [0,d_F(x_1,x_2)]$ from $x_1$ to $x_2$. Since $U$ is compact, we have
\[
\lambda_F(U)<+\infty,\quad \max_{x\in U}F^*({\dd} u|_x)<+\infty.
\]
By letting $C:=\lambda_F(U)\max_{x\in U}F^*({\dd} u|_x)$ and using \eqref{dualff*}, we have
\begin{align*}
|u(x_1)-u(x_2)|=&\left| \int^{d_F(x_1,x_2)}_0 \frac{{\dd}}{{\dd}t}u(\gamma(t))  {\dd}t\right|=\left| \int^{d_F(x_1,x_2)}_0 \langle \dot{\gamma}(t),{\dd}u\rangle  {\dd}t\right|\\
\leq &  \int^{d_F(x_1,x_2)}_0 \left|\langle \dot{\gamma}(t),{\dd}u\rangle\right|  {\dd}t\leq  \lambda_F(U) \int^{d_F(x_1,x_2)}_0 F(\dot{\gamma}(t)) F^*({\dd} u) {\dd}t\\=&  C \, d_F(x_1,x_2),
\end{align*}
which implies $C^\infty_0(M)\subset \Lip(M)$.
\end{proof}

By Rademacher's theorem and local coordinate systems, it follows that every Lipschitz function on $(M,F,\m)$ is differentiable $\m$-a.e. In addition, we have the following approximation result which is useful in our proofs.

\begin{proposition}\label{lipsconverppax-prop}
Let $(M,F,\m)$ be a forward complete $\FMMM$. For every $u\in \Lip_0(M)$,   there exists a sequence of smooth Lipschitz functions $u_l\in C^\infty_0(M)\cap \Lip_0(M)$  such that
\[
\supp u\cup \supp u_l\subset \mathscr{K},\quad |u_l|\leq C,\quad F^*({\dd} u_l)\leq   \dil(u_l)\leq C,\quad u_l\rightrightarrows u,\quad \|u_l-u \|_{W^{1,p}_{\m}}\rightarrow 0,
\]
where $\mathscr{K}\subset M$ is a compact set, $C>0$ is  a constant and ``\,$\rightrightarrows$" means uniform convergence. Moreover, if $u$ is nonnegative, so are $u_l$'s.
\end{proposition}
\begin{proof} Although the proof is standard,  we provide it for readers' convenience.
	Let $\mathbb{B}^n$ denote the  standard Euclidean unit open ball in $\mathbb{R}^n$.  The proof is divided into two steps.
	
	\smallskip
	
	\textit{Step 1.} We show that for every $C$-Lipschitz function $h\in C_0(\mathbb{B}^n)$ and every $\epsilon>0$,  there exists a $C$-Lipschitz function $h_\epsilon\in C^\infty_0(\mathbb{B}^n)$ such that
	\[
	h_\epsilon\rightrightarrows h,\quad \|h_\epsilon -h\|_{W^{1,p}(\mathbb{B}^n)}\rightarrow 0,\quad \text{as }\epsilon\rightarrow 0^+,
	\]
	where $\|\cdot\|_{W^{1,p}(\mathbb{B}^n)}$ is the standard Sobolev norm on $\mathbb{B}^n$, see \eqref{W1p}.
	Let $h_\epsilon$
	be the convolution of $h$
	with a standard mollifier $\phi_\epsilon$, i.e., $h_\epsilon:=h*\phi_\epsilon$. Note that
	\[
	|h_\epsilon(x)-h_\epsilon(y)|\leq |h(\cdot)-h(\cdot+y-x)| * \phi_\epsilon\leq C |x-y|,\quad \forall\,x,y\in C_0(\mathbb{B}^n),
	\]
	i.e., $h_\epsilon$ is $C$-Lipschitz. The other properties of $h_\epsilon$ follow from
	the standard theory of Sobolev spaces. In particular, $h_\epsilon$ is nonnegative whenever $h\geq 0$.
	
	\smallskip
	
	\textit{Step 2.}
	Since $\supp u$ is compact, there exists a finite coordinate covering $\{(V_i,\varphi_i)\}_{i=1}^{N<\infty}$ of $\supp u$ such that $V_i\Subset M$ and $\varphi_i(V_i)=\mathbb{B}^n$. We may also assume that $\mathscr{K}:=\cup_{i=1}^N\overline{V_i}$ is a compact subset in $M$.
	Let $\{\eta_i\}_{i=1}^N$ be a smooth partition of unity subordinated to $\{V_i\}_{i=1}^N$.
	An argument similar to \cite[(6.2.3)]{BCS} yields a constant $L\geq 1$ such that for all $i\in \{1,\ldots,N\}$,
	\begin{equation}\label{normconrraoll}
	L^{-1}\leq \frac{F^*(\varphi^*\xi)}{| \xi|_i}\leq L, \ \forall \xi\in T^*\mathbb{B}^n\backslash\{0\},
	\end{equation}
	\begin{equation}\label{distconleuri}
		|\eta_i(x_1)-\eta_i(x_2)|\leq L\, d_F(x_1,x_2),\ \forall\, x_1,x_2\in M,\ \max_{x\in U}F^*({\dd} \eta_i)\leq \sqrt{L},\ \lambda_F(\mathscr{K})\leq \sqrt{L},
	\end{equation}
	and
	\begin{equation}\label{coordinateestimes}
		L^{-1} d_F(x_1,x_2)\leq |\varphi_i(x_1)-\varphi_i(x_2)|_i\leq L \, d_F(x_1,x_2),
	\end{equation}
	for every $x_1,x_2\in V_i$ and $ L^{-1}{\dd}x_i\leq \dm|_{V_i}\leq L\,{\dd}x_i,$ where $|\cdot|_i$ and ${\dd}x_i$ denote  the Euclidean norm  and the Lebesgue measure on $\varphi_i(V_i)=\mathbb{B}^n$, respectively.

	According to Lemma \ref{localtogloaLip} and \eqref{coordinateestimes}$_1$ (i.e., the first inequality in \eqref{coordinateestimes}),
	it is easy to check that $h_i:=(\eta_i u)\circ \varphi_i^{-1}$ is a $C_1$-Lipschitz function in $\mathbb{B}^n$ with compact support, where
	\[
	C_1:=\max_{1\leq i\leq N}\dil(h_i)<+\infty.
	\]
	Due to \textit{Step 1}, for each $k$ and $\epsilon>0$, there exists a $C_1$-Lipschitz function $h_{i,\epsilon}\in C^\infty_0(\mathbb{B}^n)$ such that
	\begin{equation}\label{hkeconverge}
		h_{i,\epsilon}\rightrightarrows h_i,\quad \|h_{i,\epsilon} -h_i\|_{W^{1,p}(\mathbb{B}^n)}\rightarrow 0,\quad \text{as }\epsilon\rightarrow 0^+.
	\end{equation}
	Moreover, by the uniform convergence we may assume $|h_{i,\epsilon}|\leq C_1$, by choosing a larger $C_1$ if necessary.
	
	Consider $h_{i,\epsilon}\circ\varphi_i\in C^\infty_0(V_i)$; due to \eqref{coordinateestimes}$_1$, they are $C_1L$-Lipschitz functions.
	Now set
	\[
	u_\epsilon:=\sum_{i=1}^N  h_{i,\epsilon}\circ \varphi_i\in C^\infty_0(M),
	\]
	which is a $NC_1L$-Lipschitz function  with $|u_\epsilon|\leq NC_1$ and $\supp u_\epsilon\subset \cup_{i=1}^N V_i\subset  \mathscr{K}$. Moreover, a direct calculation gives
	\begin{align*}
		\left| u_\epsilon-u   \right|&=\left| \sum_{i=1}^N  h_{i,\epsilon}\circ \varphi_i-\sum_{i=1}^N  h_{i}\circ \varphi_i \right|\leq \sum_{i=1}^N  \left| h_{i,\epsilon}\circ \varphi_i-h_{i}\circ \varphi_i  \right|,
	\end{align*}
	which together with \eqref{hkeconverge} and \eqref{coordinateestimes}$_2$ furnishes $u_\epsilon\rightrightarrows u$ and $\|u_\epsilon- u\|_{L^p_{\m}}\rightarrow 0$. In addition, \eqref{normconrraoll} yields
$$
		F^*\left( {\dd} u_\epsilon-{\dd} u  \right)\leq \sum_{i=1}^NF^*\left( {\dd} (  h_{i,\epsilon}\circ \varphi_i)- {\dd} (  h_{i}\circ \varphi_i)   \right)\leq L\sum_{i=1}^N   \left| {\dd}  h_{i,\epsilon}- {\dd} h_{i}   \right|_i,
	$$
	which combined with \eqref{hkeconverge}$_2$ and \eqref{coordinateestimes}$_2$   implies  $$\|u_\epsilon-u \|_{W^{1,p}_{\m}}\rightarrow 0.$$ Furthermore, if $u\geq 0$, then $h_i\geq 0$ and hence, $h_{i,\epsilon}\geq 0$  due to \textit{Step 1}. Therefore, $u_\epsilon\geq 0$.
	We are done by letting $\epsilon=1/l$ and $C:=NC_1L$.
\end{proof}

\section{Failure of functional inequalities: proof of Theorem \ref{mainThm1}}\label{influeofS1}

In this section we  prove Theorem \ref{mainThm1}, considering in separate subsections the items (i)-(iv). Before doing this, we present an auxiliary result which will be useful in several steps.

\begin{lemma}\label{lemma_u^p_r^p}
	Let $(M,F,\m)$ be an $n$-dimensional forward complete $\FMMM$ and let $h\geq k\geq 0$ be constants. Suppose that there exists a point $o\in M$ such that
	\[
	\mathbf{Ric}(\nabla r)\geq -(n-1)k^2, \quad \mathbf{S}(\nabla r) \geq (n-1) h,
	\]
	where $r(x):=d_F(o,x)$.
	Let $u_\alpha(x):=-e^{-\alpha r(x)}$ for  $\alpha>0$. Then the following statements hold$:$
	\begin{enumerate}[\rm (i)]
		\item \label{lemma_u^p_r^p_i} $\|u_\alpha\|_{W^{1,p}_{\m}}<+\infty$ for any  $p\in [1,\infty);$
		
		\smallskip
		
		\item \label{lemma_u^p_r^p_ii}  $\ds\int_M \frac{|u_{\alpha}|^p}{r^p} \dm <  +\infty$   for any $p \in [1, n);$
		
		\smallskip
		
		\item \label{lemma_u^p_r^p_iii}  if $k>0$ then for any $p\in [1,\infty),$
		\begin{equation}\label{boundeduandF}
			\int_M |u_\alpha|^p \dm\leq \alpha^{-1}\frac{\mathscr{I}_{\m}(o)}{p k^{n-1}},\quad  \int_M {F}^{*p}({\dd}u_\alpha) \dm\leq \alpha^{p-1}\frac{\mathscr{I}_{\m}(o)}{pk^{n-1}},
		\end{equation}
		where $\mathscr{I}_{\m}(o)$ is defined by \eqref{cocont}.
	\end{enumerate}
\end{lemma}
\begin{proof} Let $(r,y)$ be the polar coordinate system around $o\in M$. According to \eqref{st11ong}, we have
	\begin{equation}\label{h_order_fisrt}
		\hat{\sigma}_o(r,y)\leq  e^{-\tau(y)-(n-1)hr}\mathfrak{s}_{-k^2}^{n-1}(r), \quad \forall\,r\in (0,i_y), \ \forall\,y\in S_oM.
	\end{equation}
	The proof is divided into two cases.
	
	\smallskip
	
	\textit{Case 1.}  Suppose  $k=0$, in which case  $\mathfrak{s}_{-k^2}(r)=r$ and $h\geq 0$.
	Then \eqref{h_order_fisrt} combined with  \eqref{volumeexprsso} yields
	\begin{align}
		\int_M |u_\alpha|^p \dm\leq & \int_{S_oM}  \left( \int^{i_y}_0 e^{-\tau(y)} e^{-p\alpha r-(n-1)hr}\mathfrak{s}_{-k^2}^{n-1}(r) {\dd}r  \right) {\dd}\nu_o(y)\notag\\
		\leq & \left(\int_{S_oM}e^{-\tau(y)}{\dd}\nu_o(y) \right)\left( \int^{+\infty}_0 e^{-p\alpha r-(n-1)hr}\mathfrak{s}_{-k^2}^{n-1}(r) {\dd}r \right)\nonumber\\=&   \mathscr{I}_{\m}(o)  \int^{+\infty}_0 e^{-p\alpha r-(n-1)hr} r^{n-1}{\dd}r\label{numcontroll}\\
		\leq &  \mathscr{I}_{\m}(o) \int^{+\infty}_0e^{-p\alpha r} r^{n-1} {\dd}r \\=&\mathscr{I}_{\m}(o) \frac{(n-1)!}{(p \alpha)^n} < +\infty.\notag
	\end{align}
	In addition, by the eikonal equation $F^*({\dd}r)=1$ we have ${F}^{*}({\dd}u_\alpha)=\alpha e^{-\alpha r}=\alpha|u_\alpha|$, which  implies
	\begin{align}\label{Fstarplimited_Kzero}
		\int_M {F}^{*p}({\dd}u_\alpha) \dm=\alpha^p \int_M|u_\alpha|^p\dm \leq \mathscr{I}_{\m}(o) \frac{(n-1)!}{p^n\alpha^{n-p}} < + \infty.
	\end{align}
	Hence, \eqref{lemma_u^p_r^p_i} follows. Furthermore, when $p\in [1,n)$,
	a similar argument yields
	\begin{align*}
		\int_M \frac{|u_\alpha|^p}{r^p} \dm =& \int_{M}  e^{- p \alpha r} r^{-p} \dm \leq  \mathscr{I}_{\m}(o) \int^{+\infty}_{0}  e^{- p \alpha r -(n-1)hr} r^{n-1-p} {\dd}r
		\\ \leq &\mathscr{I}_{\m}(o) \int^{+\infty}_{0}  e^{- p \alpha r} r^{n-1-p} {\dd}r= \mathscr{I}_{\m}(o) \frac{\Gamma(n-p)}{(p \alpha)^{n-p}}<+\infty,
	\end{align*}
	which gives \eqref{lemma_u^p_r^p_ii}.

	\smallskip
	
	\textit{Case 2.}  Suppose $k>0$, in which case $\mathfrak{s}_{-k^2}(r)=\frac{\sinh kt}{k}$.
	A computation analogous to  \eqref{numcontroll} together with $h \geq k$ furnishes
	\begin{align*}
		\int_M |u_\alpha|^p \dm\leq & \left(\int_{S_oM}e^{-\tau(y)}{\dd}\nu_o(y) \right)\left( \int^{+\infty}_0 e^{-p\alpha r-(n-1)hr}\mathfrak{s}_{-k^2}^{n-1}(r) {\dd}r \right)\label{removeiyres}\\
		\leq &\frac{ \mathscr{I}_{\m}(o)}{k^{n-1}}  \int^{+\infty}_0 e^{-p\alpha r-(n-1)hr+(n-1)kr}{\dd}r\leq \frac{ \mathscr{I}_{\m}(o)}{k^{n-1}} \int^{+\infty}_0e^{-p\alpha r}{\dd}r=\frac{\mathscr{I}_{\m}(o)}{p\alpha k^{n-1}},\notag
	\end{align*}
	which also yields
	\begin{align*}
		\int_M {F}^{*p}({\dd}u_\alpha) \dm=\alpha^p \int_M|u_\alpha|^p\dm\leq \frac{\mathscr{I}_{\m}(o)}{pk^{n-1}}\alpha^{p-1}.
	\end{align*}
	Therefore, both \eqref{lemma_u^p_r^p_i} and \eqref{lemma_u^p_r^p_iii} hold.
	To prove \eqref{lemma_u^p_r^p_ii}, we need the following estimates of $\hat{\sigma}_o(r,y)$.
	On the one hand, \eqref{h_order_fisrt}  yields
	\[
	\hat{\sigma}_o(r,y)
	\leq  \frac{e^{-\tau(y)-(n-1)(h-k)r}}{(2k)^{n-1}}, \quad \forall\,r\in (0,i_y), \ \forall\,y\in S_oM.
	\]
	On the other hand,   (\ref{volumerzero}) guarantees a small $\varepsilon_0\in (0,\mathfrak{i}_o)$ such that
	\[
	\hat{\sigma}_o(r,y) \leq 2 e^{-\tau(y)}r^{n-1}, \quad \forall\,  r \in (0, \varepsilon_0), \ \forall\,y\in S_oM.
	\]
	Hence, for any $p\in [1, n)$, an argument similar to \eqref{numcontroll} yields
	\begin{align*}
		\int_M \frac{|u_\alpha|^p}{r^p}\dm
		& =  \int_{B^+_{o}(\varepsilon_0)} e^{- p \alpha r}r^{-p}\dm + \int_{M \setminus B^+_{o}(\varepsilon_0)}  e^{- p \alpha r} r^{-p} \dm \\
		& \leq  2\mathscr{I}_{\m}(o) \int^{\varepsilon_0}_0   r^{n-1 - p} {\dd}r +\frac{\mathscr{I}_{\m}(o)}{(2k)^{n-1}}  \int^{+\infty}_{\varepsilon_0}  e^{- p \alpha r -(n-1)(h-k)r}  r^{-p} {\dd}r  \\
		& \leq  \frac{2 \mathscr{I}_{\m}(o)}{n - p} \varepsilon_0^{n-p} + \frac{\mathscr{I}_{\m}(o)}{(2k)^{n-1}}\varepsilon_0^{-p} \int^{+\infty}_{\varepsilon_0} e^{- p \alpha r } {\dd}r  < +\infty,
	\end{align*}
	which concludes the proof.
\end{proof}

\begin{remark}
Under the assumption of Lemma \ref{lemma_u^p_r^p}, one can prove in a standard manner that  $u_\alpha\in W_0^{1,p}(M,\m)$; since this fact is not used explicitly, we omit its proof.
\end{remark}

\subsection{Vanishing first eigenvalue of $p$-Laplacian}
In this subsection we will study the first eigenvalue of the $p$-Laplacian in the Finsler setting.
According to \cite{ShenLecture},    the {\it divergence} of a smooth vector field $X$ with respect to $\m$ is defined as
\[
\Div_{\m}(X)\dm:={\dd}(X\rfloor \dm).
\]
Given $p\in (1,\infty)$ and $u\in C^\infty(M)$, the {\it $p$-Laplacian  of $u$ } (with respect to $\m$) is defined by
\begin{equation}\label{plapc1}
 \Delta_{p, \m} u := \Div_{\m}\left(F^{p-2}(\nabla  u)  \nabla  u\right)
 \end{equation}
on $\mathcal {U}:=\{x\in M\,:\, {\dd}{u}|_x\neq 0\}$. Moreover, if $u\in W^{1,p}_{\loc}(M)$,
  the {\it $p$-Laplacian  of $u$ in the weak sense} (with respect to $\m$) is defined by
\begin{equation}\label{plapc2}
\int_M v \Delta_{p, \m} u \dm = - \int_M F^{p-2}(\nabla u)\langle\nabla u, {\dd}v \rangle \dm,
\end{equation}
 for any $v\in C^{\infty}_0 (M)$. In particular, the Laplacian $\Delta_{\m} u$ is exactly $\Delta_{2, \m}u$.




Let $p\in [1,\infty)$ and $(M,F,\m)$ be an $n$-dimensional forward complete $\FMMM$; for further use, let us consider
\[
\Lambda_p(\cdot): = \frac{\ds\int_M {F}^{*}({\dd}\,\cdot\,)^p  \dm}{\ds\int_M |\cdot|^p \dm}.
\]

\begin{lemma}\label{Eigenvaluelemma} Let $(M,F,\m)$ be an $n$-dimensional forward complete $\FMMM$ and let $h\geq k\geq 0$ be constants. Assume that there exists some point $o\in M$ such that
\[
\mathbf{Ric}(\nabla r)\geq -(n-1)k^2, \quad \mathbf{S}(\nabla r) \geq (n-1) h,
\]
with $r(x):=d_F(o,x)$. Then
for any $p\in [1, \infty)$,
\begin{equation}\label{hardfail_forEigenvalue}
\inf_{u\in C^\infty_0(M)\backslash\{0\}} \Lambda_p(u)\leq \liminf_{\alpha\rightarrow 0^+}\Lambda_p(u_{\alpha}),
\end{equation}
where  $u_{\alpha}:=-e^{-\alpha r}$ for $\alpha>0$.
\end{lemma}
\begin{proof} We divide the proof into three steps.

\smallskip

\textit{Step 1.} In this step, we are going to show that
\[
\inf_{u\in C^\infty_0(M)\backslash\{0\}} \Lambda_p(u) = \inf_{u \in \Lip_0(M)\backslash\{0\}} \Lambda_p(u).
\]
Since $C^\infty_0(M) \subset \Lip_0(M)$, it is obvious that
\begin{equation*}\label{EigenJuJv_1}
\inf_{u \in \Lip_0(M)\backslash\{0\}} \Lambda_p(u)\leq \inf_{u\in C^\infty_0(M)\backslash\{0\}} \Lambda_p(u).
\end{equation*}
For the reverse inequality, it suffices to show that for any $u\in \Lip_0(M)\backslash\{0\}$, there exists a sequence $u_l\in C^\infty_0(M)\backslash\{0\}$  such that
\begin{equation}\label{Eigenreverinverhar1}
\lim_{l\rightarrow \infty} \Lambda_p(u_l) = \Lambda_p(u).
\end{equation}

Given $u \in \Lip_0(M)$,  Proposition \ref{lipsconverppax-prop} provides
  a sequence of  functions $(u_l) \subset C^\infty_0(M)\cap \Lip_0(M)$  satisfying
\begin{equation}\label{threeresults_forEigenvalue}
\supp u\cup \supp u_l\subset \mathscr{K}, \quad u_l\rightrightarrows u,\quad \|u_l - u \|_{W^{1,p}_{\m}}\rightarrow 0,
\end{equation}
where $\mathscr{K}$ is some compact set in $M$.
The uniform convergence in \eqref{threeresults_forEigenvalue} implies
\begin{equation}\label{lim_u/r_forEigenvalue}
\lim_{l\rightarrow \infty} \int_{M} |u_l|^p \dm  = \int_{M}  \lim_{l\rightarrow \infty}|u_l|^p \dm = \int_{M}  |u|^p  \dm.
\end{equation}
The $W^{1,p}$-convergence in (\ref{threeresults_forEigenvalue}) combined with \eqref{W1p} yields
\begin{equation}\label{F_du_l_du_forEigenvalue}
\|{\dd} u_l - {\dd} u\|_{L^p_{\m}}\leq   \|u_l - u\|_{W^{1,p}_{\m}}  \rightarrow 0, \quad  \text{ as } l\rightarrow \infty.
 \end{equation}
As  $\supp u\cup \supp u_l\subset \mathscr{K}$ and  $\lambda_{F^*}(\mathscr{K}) < +\infty$, we have
\begin{align}\label{F_du_du_l_forEigenvalue}
\lim_{l\rightarrow \infty} \|{\dd} u - {\dd} u_l\|_{L^p_{\m}}
\leq \lambda_{F^*}(\mathscr{K})\lim_{n\rightarrow \infty} \|{\dd} u_l - {\dd} u\|_{L^p_{\m}}=0,
\end{align}
Now the triangle inequality for $F^*$ implies that
\begin{equation*}
  \|{\dd} u_l\|_{L^p_{\m}}  \leq \|{\dd} u_l - {\dd} u\|_{L^p_{\m}} + \|{\dd} u \|_{L^p_{\m}}, \quad  \|{\dd} u \|_{L^p_{\m}}  \leq \|{\dd} u - {\dd} u_l\|_{L^p_{\m}} + \|{\dd} u_l\|_{L^p_{\m}},
\end{equation*}
which combined with  (\ref{F_du_l_du_forEigenvalue}) and (\ref{F_du_du_l_forEigenvalue}) yields
\begin{equation}\label{Flimivanishes_forEigenvalue}
 \int_{M} F^{* p}({\dd} u) \dm =\|{\dd}u\|^p_{L^p_{\m}}= \lim_{l\rightarrow \infty}\|{\dd}u_l\|^p_{L^p_{\m}}= \lim_{l\rightarrow \infty} \int_{M} F^{* p}( {\dd} u_l) \dm.
\end{equation}
Then (\ref{Eigenreverinverhar1}) follows directly by (\ref{lim_u/r_forEigenvalue}) and \eqref{Flimivanishes_forEigenvalue}.

\smallskip

\textit{Step 2.} In this step, we prove that for every $u_{\alpha}:=-e^{-\alpha r}$ with  $\alpha>0$, there exists a sequence $(v_{\alpha, l}) \subset \Lip_0(M)$ such that
\begin{equation}\label{FpStep2main_forEigenvalue}
\Lambda_p (u_{\alpha}) = \lim_{l\rightarrow \infty} \Lambda_p(v_{\alpha, l}).
\end{equation}

If $\sup_{x\in M}r(x)<+\infty$, then $M$ is compact and $u_{\alpha}\in  \Lip_0(M)$. Hence, \eqref{FpStep2main_forEigenvalue} is trivial by choosing $v_{\alpha, l}=u_\alpha$.
In the sequel, we consider the case when $\sup_{x\in M}r(x)=+\infty$, i.e., $M$ is noncompact.
For every small $\varepsilon\in (0,\min\{1,\mathfrak{i}_o\})$ and $\delta\in (0,e^{-\alpha })$, set
\[
u_{\alpha, \varepsilon}(x):=
\begin{cases}
-e^{-\alpha\varepsilon}, &\text{ if }  r(x) < \varepsilon;\\
\\
-e^{-\alpha r(x)}, &\text{ if }   r(x) \geq \varepsilon,
\end{cases}\qquad u_{\alpha, \varepsilon,\delta}:=\min\{0, u_{\alpha, \varepsilon} +\delta\}.
\]

On the one hand, we claim  that $u_{\alpha, \varepsilon,\delta}\in \Lip_0(M)$. In fact,
it is easy to check that
\[
\supp u_{\alpha, \varepsilon,\delta}=\overline{\{u_{\alpha, \varepsilon,\delta}<0\}}=\overline{B^+_o\left( -\frac{\ln\delta}{\alpha}\right)},
\]
which is a compact set. In addition, since the construction implies
\[
 F^{*}(\pm {\dd} u_{\alpha, \varepsilon,\delta})\leq \max_{x\in \overline{B^+_o( -\alpha^{-1}\ln\delta)}} \max\{F^{*}(\pm{\dd} u_{\alpha}|_x)\}=:C<+\infty,
\]
 an argument similar to the proof of Lemma \ref{localtogloaLip} yields
\[
\left|u_{\alpha, \varepsilon,\delta}(x_1)-u_{\alpha, \varepsilon,\delta}(x_2)\right|\leq C d_F(x_1,x_2), \quad \forall\,x_1,x_2\in M,
\]
which proves the claim.

On the other hand, set $\tilde{u}_{\alpha, \varepsilon}:=u_{\alpha, \varepsilon,\varepsilon}\in \Lip_0(M)$ for every $\varepsilon\ll 1$ and $v_{\alpha,l}:=\tilde{u}_{\alpha, 1/l}\in \Lip_0(M)$.
For \eqref{FpStep2main_forEigenvalue}, we only need to show
\begin{equation}\label{FpStep2main_forEigenvalue2}
\Lambda_p (u_{\alpha}) = \lim_{\varepsilon\rightarrow 0^+} \Lambda_p(\tilde{u}_{\alpha, \varepsilon}).
\end{equation}
The above construction implies the following pointwise convergence
\begin{equation}\label{varisonpoiconve_forEigenvalue}
  |\tilde{u}_{\alpha, \varepsilon}|\nearrow |u_\alpha|, \quad F^*({\dd}\tilde{u}_{\alpha, \varepsilon})\nearrow F^*({\dd}u_\alpha), \quad \text{  for $\m$-a.e.,}
\end{equation}
as $\varepsilon\to 0^+.$
Moreover, Lemma \ref{lemma_u^p_r^p}/\eqref{lemma_u^p_r^p_i} implies that $|u_{\alpha}|,F^{*}({\dd} u_{\alpha})\in L^p(M,\m)$, which together with the dominated convergence theorem and  \eqref{varisonpoiconve_forEigenvalue} yields
\begin{equation}\label{FpStep2part1_forEigenvalue}
\lim_{\varepsilon\rightarrow 0^+} \int_{M} |\tilde{u}_{\alpha, \varepsilon}|^p \dm = \int_{M} |u_{\alpha}|^p \dm,\quad \lim_{\varepsilon \rightarrow 0^+} \int_{M} F^{*p}( {\dd}\tilde{u}_{\alpha, \varepsilon}) \dm =   \int_{M} F^{*p}( {\dd} u_{\alpha}) \dm;
\end{equation}
then the limit in  \eqref{FpStep2main_forEigenvalue2} follows.

\smallskip

\textit{Step 3.} We now conclude the proof. For any $\alpha>0$ and $l\in \mathbb{N}$,
 \textit{Steps 1} and \textit{2}   yield
\[
\inf_{u\in C^\infty_0(M)} \Lambda_p(u)=\inf_{u\in \Lip_0(M)} \Lambda_p(u)\leq \Lambda (v_{\alpha,l}).
\]
By letting first $l\rightarrow \infty$ and then  $\alpha\rightarrow 0^+$,  and using \eqref{FpStep2main_forEigenvalue}, we get
\[
\inf_{u\in C^\infty_0(M)} \Lambda_p(u)\leq \liminf_{\alpha\rightarrow 0^+}\lim_{l\rightarrow \infty } \Lambda_p(v_{\alpha,l})=\liminf_{\alpha\rightarrow 0^+}\Lambda_p (u_{\alpha}),
\]
which is exactly \eqref{hardfail_forEigenvalue}.
\end{proof}

In view of \eqref{plapc1} and \eqref{plapc2}, the {\it first eigenvalue} of the $p$($>1$)-Laplace operator $\Delta_{p,\m}$  is defined as
\begin{equation} \label{lambda1}
 \lambda_{1, p,\m} (M)  = \inf_{u\in C^\infty_0(M)\backslash \{ 0 \} } \frac{\ds\int_{M} {F}^{*p }({\dd}u) \dm}{\ds\int_{M} |u|^p \dm}=\inf_{u\in C^\infty_0(M)\backslash \{ 0 \} } \Lambda_p(u).
\end{equation}
We have the following result, which proves a slightly more general version as Theorem \ref{mainThm1}/(i):

\begin{theorem}\label{Thm_Eigenvalue}
Let $(M,F,\m)$ be an $n$-dimensional forward complete $\FMMM$ and let $h\geq k\geq 0$ be constants. If there exists some point $o\in M$ such that
\[
\mathbf{Ric}(\nabla r)\geq -(n-1)k^2, \quad \mathbf{S}(\nabla r) \geq (n-1) h,
\]
where $r(x):=d_F(o,x)$,
then $\lambda_{1,p, \m} (M)  = 0$ for all $p\in (1,\infty)$.
\end{theorem}
\begin{proof}
For every $\alpha > 0$, let $u_{\alpha} :=-e^{-\alpha r}$.
 Since $F^{*}({\dd}r) = 1$, we have
\begin{equation}\label{bascffa}
\ds\int_{M} {F}^{*p }( {\dd}u_{\alpha}) \dm
 = \alpha^p \ds\int_{M}  e^{-p \alpha r} \dm,\quad
\ds\int_{M} |u_{\alpha}|^p(x) \dm = \ds\int_{M}  e^{- p \alpha r} \dm.
\end{equation}
Then  (\ref{lambda1}) combined with Lemma \ref{Eigenvaluelemma} and \eqref{bascffa} yields
\begin{align*}
\lambda_{1,p,  \m} (M)& = \inf_{u\in C^\infty_0(M)\backslash \{ 0 \} } \Lambda_p(u)\leq\liminf_{\alpha\rightarrow 0^+} \Lambda_p(u_\alpha)= \liminf_{\alpha\rightarrow 0^+}  \alpha^p =0,
\end{align*}
which concludes the proof.
\end{proof}

\begin{corollary}\label{basicwanisheigven}
Let $(M,F,\m)$ be an $n$-dimensional forward complete $\FMMM$ with
\[
\mathbf{Ric} \geq -(n-1)k^2,\quad \mathbf{S}  \geq (n-1) h,
\]
for some constants $h\geq k\geq 0$. Then $\lambda_{1,p, \m} (M)  = 0$ for all $p\in(1,\infty)$.
\end{corollary}

\begin{remark} (i) Consider a noncompact Riemannian manifold equipped with the canonical Riemannian measure satisfying $\mathbf{Ric}\geq  0$.   Corollary  \ref{basicwanisheigven} implies the first eigenvalue of the Laplacian vanishes, which has been observed first by  Cheng--Yau \cite[Proposition 9]{Cheng}.
	
	(ii) Theorem \ref{Thm_Eigenvalue} and Corollary \ref{basicwanisheigven} have been proved by Krist\'aly \cite[Theorem 1.3]{Kris} in the special case when $(M,F,\m)$ is the Funk metric on the unit Euclidean ball $\mathbb B^n$ endowed with the Busemann--Hausdorff measure.
\end{remark}

 The following result is a direct consequence of Corollary \ref{basicwanisheigven}.
\begin{corollary}\label{firsteginvees} Let $(M,F)$ be an $n$-dimensional Cartan--Hadamard manifold with constant flag curvature $-k^2<0$. For each $p\in (1,\infty)$, the following statements hold$:$
\begin{enumerate}[\rm (a)]
\item If there exists a positive smooth measure $\m$ such that $\mathbf{S} \geq (n-1)|k|$,
then $\lambda_{1, p,\m} (M)  = 0$.

\smallskip

\item\label{poeigve} If there exists a positive smooth measure $\m$ such that $\lambda_{1, p,\m} (M) >0$, then $\ds\inf_{y\in SM}\mathbf{S}(y)<(n-1)|k|$.
\end{enumerate}
\end{corollary}

\begin{remark} Let $(\mathbb H^n,g)$ be  the $n$-dimensional hyperbolic space with constant sectional curvature ${\bf K}=-k^2<0$, endowed with the canonical Riemannian measure $\m_{\rm can}$;  one has the optimal spectral gap
\[	\lambda_{1, p,\m_{\rm can}} (\mathbb H^n)=\frac{(n-1)^{2p}k^{2p}(p-1)^p}{p^{2p}}>0, \]
see  McKean \cite{Mck} for $p=2$, and Farkas--Kaj\'ant\'o--Krist\'aly \cite{FKK} for generic $p>1$. 	
The fact that $\mathbf{S}=0$ in this case aligns perfectly with
Corollary \ref{firsteginvees}/\eqref{poeigve}.
\end{remark}




\subsection{Failure of Hardy inequality}
In this subsection, we investigate the Hardy inequality in the Finsler framework.
As before, let  $(M,F,\m)$ be  an $n$-dimensional forward complete $\FMMM$ and let $o\in M$ be a fixed point.
Given $p\in (1,\infty)$, set
\begin{equation}\label{Hardyifuncitonal}
 \mathcal{J}_{o,p}(u): =  \frac{\ds\int_M {F}^{*p}({\dd}u) \dm}{\ds\int_M \frac{|u|^p}{r^p}\dm},
\end{equation}
where $r(x):=d_F(o,x)$. In order to establish the $L^p$-Hardy inequality, it is equivalent to find
\[
\inf_{u\in C^\infty_0(M)\backslash\{0\}}\mathcal{J}_{o,p}(u).
\]
Due to  Lemma \ref{dulemma2}/\eqref{firtesk2},  provided that $p\in [n,\infty)$, for every $ u\in C^\infty_0(M)$ with $u(x)=c\neq 0$ for every $x\in B^+_o(r)$ and for some $r>0$, we have that
\begin{equation}\label{hardfailn>p}
\int_M {F}^{*p}({\dd}u) \dm<+\infty, \quad \int_M \frac{|u|^p}{r^p}\dm=+\infty,
\end{equation}
which implies $$\inf_{u\in C^\infty_0(M)\backslash\{0\}}\mathcal{J}_{o,p}(u)=0,$$ i.e., the Hardy inequality fails. Therefore, it is enough to consider the case when $p\in (1,n)$.

\begin{lemma}\label{Hardyfinallemma}
Let $(M,F,\m)$ be an $n$-dimensional forward complete $\FMMM$ and let $h\geq k\geq 0$ be constants. Assume there exists a point $o\in M$ such that
\[
\mathbf{Ric}(\nabla r)\geq -(n-1)k^2, \quad \mathbf{S}(\nabla r) \geq (n-1) h,
\]
with $r(x):=d_F(o,x)$. Then for any  $p\in (1, n)$,
\begin{equation}\label{hardfail}
\inf_{u\in C^\infty_0(M)\backslash\{0\}} \mathcal{J}_{o,p}(u)\leq \liminf_{\alpha \rightarrow 0^+} \mathcal{J}_{o,p}(u_{\alpha}),
\end{equation}
where $u_{\alpha}:=-e^{-\alpha r}$ for $\alpha>0,$ and $\mathcal{J}_{o,p}$ is defined by \eqref{Hardyifuncitonal}.
\end{lemma}
\begin{proof} The proof being similar to  Lemma \ref{Eigenvaluelemma}, we only sketch it,   dividing into three steps.

\smallskip

\textit{Step 1.} We show that
\[
\inf_{u\in C^\infty_0(M)\backslash\{0\}} \mathcal{J}_{o,p}(u) = \inf_{u \in \Lip_0(M)\backslash\{0\}} \mathcal{J}_{o,p}(u).
\]
In view of  $C^\infty_0(M) \subset \Lip_0(M)$,  it suffices to show that for any $u\in \Lip_0(M)$, there exists a sequence $u_l\in C^\infty_0(M)\backslash\{0\}$  such that
\begin{equation}\label{reverinverhar1}
\lim_{l\rightarrow \infty} \mathcal{J}_{o,p}(u_l) = \mathcal{J}_{o,p}(u).
\end{equation}
Given any $u \in \Lip_0(M)$,  Proposition \ref{lipsconverppax-prop} guarantees the existence of
  a sequence  $(u_l) \subset C^\infty_0(M)\cap \Lip_0(M)$  satisfying
\begin{equation}\label{threeresults}
\supp u\cup\supp u_l\subset \mathscr{K}, \quad u_l\rightrightarrows u,\quad \|u_l - u \|_{W^{1,p}_{\m}}\rightarrow 0,
\end{equation}
where $\mathscr{K}\subset M$ is a compact set.
The uniform convergence in (\ref{threeresults}) combined with Lemma \ref{dulemma2}/\eqref{firtesk1} yields
\begin{equation*}\label{lim_u/r}
\lim_{l\rightarrow \infty} \int_{M} \frac{|u_l|^p}{r^p} \dm =\lim_{l\rightarrow \infty} \int_{\mathscr{K}} |u_l|^p {\dd}\mu = \int_{\mathscr{K}}  \lim_{l\rightarrow \infty}|u_l|^p {\dd}\mu = \int_{\mathscr{K}} |u|^p {\dd}\mu= \int_{M} \frac{|u|^p}{r^p} \dm,
\end{equation*}
where ${\dd}\mu = r^{-p} \dm$. Then  \eqref{reverinverhar1} follows because   \eqref{Flimivanishes_forEigenvalue} remains valid.

\smallskip

\textit{Step 2.} We are going to show that for every $u_{\alpha}:=-e^{-\alpha r}$ with  $\alpha>0$, there exists a sequence $(v_{\alpha, l}) \subset \Lip_0(M)$ such that
\begin{equation}\label{FpStep2main}
\mathcal{J}_{o,p} (u_{\alpha}) = \lim_{l\rightarrow \infty} \mathcal{J}_{o,p} (v_{\alpha, l}).
\end{equation}
It is enough to discuss the case $\sup_{x\in M}r(x)=+\infty$; otherwise, $M$ is compact and \eqref{FpStep2main} is obvious by choosing $v_{\alpha, l}=u_{\alpha}\in \Lip_0(M)$.
For every small $\varepsilon>0$, set
\[
u_{\alpha, \varepsilon}:=
\begin{cases}
-e^{-\alpha\varepsilon}, &\text{ if }  r < \varepsilon,\\
\\
-e^{-\alpha r}, &\text{ if }   r \geq \varepsilon,
\end{cases}
\qquad \tilde{u}_{\alpha, \varepsilon}:=\min\{0, u_{\alpha, \varepsilon} +\varepsilon\},\qquad v_{\alpha,l}:=\tilde{u}_{\alpha, 1/l}.
\]
Then $\tilde{u}_{\alpha, \varepsilon}\in  \Lip_0(M)$ follows by the same argument as in \textit{Step 2} of the proof of  Lemma \ref{Eigenvaluelemma}. Moreover,
  the dominated convergence theorem combined with  \eqref{varisonpoiconve_forEigenvalue} and Lemma \ref{lemma_u^p_r^p}/\eqref{lemma_u^p_r^p_i}\eqref{lemma_u^p_r^p_ii} yields
\[
\lim_{\varepsilon\rightarrow 0^+} \int_{M} \frac{|\tilde{u}_{\alpha, \varepsilon}|^p}{r^p} \dm = \int_{M} \frac{|u_\alpha|^p}{r^p} \dm, \quad \lim_{\varepsilon \rightarrow 0^+} \int_{M} F^{*p}( {\dd}\tilde{u}_{\alpha, \varepsilon}) \dm =   \int_{M} F^{*p}( {\dd} u_{\alpha}) \dm,
\]
which implies
$\lim_{\epsilon\rightarrow 0^+} \mathcal{J}_{o,p} (\tilde{u}_{\alpha, \varepsilon})=\mathcal{J}_{o,p} (u_{\alpha})$, concluding the proof of
 \eqref{FpStep2main}.

\smallskip

\textit{Step 3.}
By \textit{Steps 1} and \textit{2}, we obtain for every $l\in \mathbb N$ that
\[
\inf_{u\in C^\infty_0(M)} \mathcal{J}_{o,p}(u)=\inf_{u\in \Lip_0(M)} \mathcal{J}_{o,p}(u)\leq \mathcal{J}_{o,p}(v_{\alpha,l}),
\]
which together with \eqref{FpStep2main} furnishes
\[
\inf_{u\in C^\infty_0(M)\backslash\{0\}} \mathcal{J}_{o,p}(u)\leq \liminf_{\alpha\rightarrow 0^+}\lim_{l\rightarrow \infty }\mathcal{J}_{o,p}(v_{\alpha,l})=\liminf_{\alpha\rightarrow 0^+}\mathcal{J}_{o,p}(u_\alpha),
\]
 completing the proof.
\end{proof}


\begin{theorem}\label{ThmHardyFails1}
Let $(M,F,\m)$ be  an $n$-dimensional forward complete $\FMMM$ and let $k,h$ be two constants with $h\geq k\geq 0$ and $h>0$.
Suppose that there exists a point $o\in M$  such that
\[
\mathbf{Ric}(\nabla r)\geq -(n-1)k^2 , \quad \mathbf{S}(\nabla r)\geq (n-1){h},
\]
where $r(x):=d_F(o,x)$. Then the $L^p$-Hardy inequality fails for all $p\in (1, n)$, i.e.,
\begin{equation}\label{failhARD}
\inf_{u\in C^\infty_0(M)\backslash\{0\}}\frac{\ds\int_M {F}^{*p}({\dd}u) \dm}{\ds\int_M \frac{|u|^p}{r^p}\dm}=0.
\end{equation}
\end{theorem}
\begin{proof} Given $\alpha\in (0,1]$, set $u_\alpha(x):=-e^{-\alpha r(x)}$. Let $(r,y)$ denote the polar coordinate system around $o$.
Due to \eqref{volumerzero}, there exists an $\varepsilon \in (0,\mathfrak{i}_o)$ such that
\begin{equation}\label{sigmasmall}
\hat{\sigma}_o(r,y)\geq \frac12 e^{-\tau(y)}r^{n-1},\quad \forall\,r\in (0,\varepsilon),\ \forall\,y\in S_oM.
\end{equation}
Therefore, a direct calculation yields
\begin{align}
\int_M \frac{|u_\alpha|^p}{r^p}\dm& \geq \int_{S_oM}\int^{\varepsilon}_0 \frac{e^{-\alpha p r}}{r^p} \frac12 e^{-\tau(y)}r^{n-1}{\dd}\nu_o(y){\dd}r
\geq  \frac{\mathscr{I}_{\m}(o)e^{- \alpha p \varepsilon}}{2}\int^{\varepsilon}_0   {r^{n-p-1}}{\dd}r\notag\\
& = \frac{\mathscr{I}_{\m}(o)}{2e^{\alpha p\varepsilon}(n-p)} \varepsilon^{n-p}.\label{standintegr_Knegative}
\end{align}
The rest of the proof is divided into two cases.

\smallskip

\textit{Case 1}: $k>0$.  For this case, combining \eqref{Hardyifuncitonal}, \eqref{hardfail}, \eqref{standintegr_Knegative}, and (\ref{boundeduandF}) yields
\begin{align*}
\inf_{u\in C^\infty_0(M)\backslash\{0\}}\frac{\ds\int_M {F}^{*p}({\dd}u) \dm}{\ds\int_M \frac{|u|^p}{r^p}\dm}\leq &
\liminf_{\alpha\rightarrow 0^+}\frac{\ds\int_M {F}^{*p}({\dd}u_\alpha) \dm}{\ds\int_M \frac{|u_\alpha|^p}{r^p}\dm}\\ \leq &  \liminf_{\alpha\rightarrow 0^+}\frac{2 e^{\alpha p\varepsilon}(n-p) \alpha^{p-1} }{pk^{n-1}\varepsilon^{n-p}}\\=& 0,
\end{align*}
which establishes \eqref{failhARD}.

\smallskip

\textit{Case 2}:  $k=0$.
Due to    ${F}^{*}({\dd}u_\alpha)=\alpha e^{-\alpha r}=\alpha|u_\alpha|$,  an argument similar to \eqref{st11ong}  yields
\begin{align}
\nonumber \int_M {F}^{*p}({\dd}u_\alpha) \dm=&\alpha^p \int_M|u_\alpha|^p\dm \\ \leq & \alpha^p \int_{S_oM}  \left( \int^{+ \infty}_0 e^{-\tau(y)} e^{-p\alpha r-(n-1)hr}r^{n-1} {\dd}r  \right) {\dd}\nu_o(y) \notag \\
= &  \nonumber  \alpha^p \mathscr{I}_{\m}(o)   \int^{+\infty}_0  e^{-p\alpha r-(n-1)hr}r^{n-1} {\dd}r \\ =&     \alpha^p \mathscr{I}_{\m}(o) \frac{(n-1)!}{(\alpha p+(n-1)h)^n}< +\infty.\label{Fstar_Kzero}
\end{align}
Then \eqref{hardfail} together with   \eqref{standintegr_Knegative} and \eqref{Fstar_Kzero} implies
\begin{align*}
\inf_{u\in C^\infty_0(M)\backslash\{0\}}\frac{\ds\int_M {F}^{*p}({\dd}u) \dm}{\ds\int_M \frac{|u|^p}{r^p}\dm}\leq&\liminf_{\alpha\rightarrow 0^+}\frac{\ds\int_M {F}^{*p}({\dd}u_\alpha) \dm}{\ds\int_M \frac{|u_\alpha|^p}{r^p}\dm}\\ \leq&  \liminf_{\alpha\rightarrow 0^+}\frac{2 e^{\alpha p\varepsilon} (n-p)(n-1)! \,\alpha^p  }{(\alpha p+(n-1)h)^n\varepsilon^{n-p}} \\=& 0,
\end{align*}
which concludes the proof.
\end{proof}

A direct consequence of Theorem \ref{ThmHardyFails1} reads as follows.

\begin{corollary}\label{failhardcor1}
	Let $(M,F,\m)$ be  an $n$-dimensional forward complete $\FMMM$ with
	\[
	\mathbf{Ric}\geq -(n-1) k^2, \quad  \mathbf{S} \geq (n-1){h}
	\]
	where  $k,h$ are two constants with $h\geq k\geq 0$ and $h>0$.
	Then  for every point $o\in M$ and every $p\in  (1, n)$,
	the $L^p$-Hardy inequality fails, i.e., \eqref{failhARD}  holds.
\end{corollary}

\begin{remark}\rm (i)
	A closer inspection of the proof of Theorem \ref{ThmHardyFails1} clearly shows that if $h=0$, the last limit  would be $+\infty$ since $p<n$. This is not surprising since in the Riemannian setting, the $L^p$-Hardy inequality is valid on every  $n$-dimensional Cartan--Hadamard manifold  equipped with the canonical Riemannian measure, due to the vanishing $S$-curvature \eqref{RiedisS}, see Yang--Su--Kong \cite{Yang} and Krist\'aly \cite{K2}.
	On the other hand, there are several Finsler manifolds satisfying the assumptions of Corollary \ref{failhardcor1}; see Section  \ref{sectionFunk} for details.
	
	(ii) In Corollary \ref{failhardcor1}, the noncompactness of $M$ is unnecessary. Indeed, if $M$ is compact, then the constant functions belong to $C^\infty_0(M)=C^\infty(M)$, which immediately furnishes \eqref{failhARD}.
\end{remark}


\subsection{Failure of Heisenberg--Pauli--Weyl uncertainty principle}
In this subsection, we study the Heisenberg--Pauli--Weyl principle.

\begin{theorem}\label{ThmFalHPW}
Let $(M,F,\m)$ be  an $n$-dimensional forward complete $\FMMM$ and let $k,h$ be two constants with $h> k\geq 0$.
Assume that there exists a point $o\in M$  such that
\[
  \mathbf{Ric}(\nabla r)\geq  -(n-1) k^2, \quad  \mathbf{S}(\nabla r)\geq (n-1){h},
\]
where $r(x):=d_F(o,x)$. Then  the Heisenberg--Pauli--Weyl uncertainty principle fails, i.e.,
\begin{equation}\label{failHPW}
\inf_{u\in C^\infty_0(M)\backslash\{0\}}\frac{\left( \ds\int_M F^*({\dd}u)^2\dm \right)\left(  \ds\int_M u^2 r^2\dm \right) }{\left( \ds\int_M u^2 \dm  \right)^2}=0.
\end{equation}
\end{theorem}
\begin{proof} We only sketch the proof as it is similar to that of Theorem \ref{ThmHardyFails1}.
For simplicity of notations, set
 \begin{equation}\label{joulast}
\mathcal {H}_{o}(u):=\frac{\left( \ds\int_M F^*({\dd}u)^2\dm \right)\left(  \ds\int_M u^2 r^2\dm \right) }{\left( \ds\int_M u^2 \dm  \right)^2}.
\end{equation}
By the same method as employed in the proof Lemma \ref{Hardyfinallemma}, one can show that
\begin{equation}\label{HPWfail}
\inf_{u\in C^\infty_0(M)\backslash\{0\}} \mathcal {H}_{o}(u)\leq \liminf_{\alpha \rightarrow 0^+} \mathcal {H}_{o}(u_{\alpha}),
\end{equation}
where $u_{\alpha}:=-e^{-\alpha r}$ for $\alpha>0$. Then it suffices to show $\lim_{\alpha \rightarrow 0^+} \mathcal {H}_{o}(u_{\alpha})=0$.

Since $h> k\geq 0$, it is easy to check that
\[
\lim_{r\rightarrow +\infty}e^{-(n-1)hr}\mathfrak{s}_{-k^2}^{n-1}(r)r^2=0,
\]
which yields a $\delta=\delta(n,k,h)>0$ such that
\begin{equation}\label{deltex2}
e^{-(n-1)hr}\mathfrak{s}_{-k^2}^{n-1}(r)r^2< 1 \quad  \text{for }r\in (\delta,+\infty).
\end{equation}
Relation \eqref{st11ong} combined with \eqref{deltex2} yields
\begin{align}
\int_M u_\alpha^2r^2\dm\leq
&
\mathscr{I}_{\m}(o)\left(  \int_{0}^{\delta}+\int_{\delta}^{+\infty}e^{-2\alpha r-(n-1)hr}\mathfrak{s}_{-k^2}^{n-1}(r)r^2{\dd}r  \right)\notag\\
\leq& \mathscr{I}_{\m}(o) \left( \mathfrak{s}_{-k^2}^{n-1}(\delta) \int^\delta_0 r^2{\dd}r+\int_{\delta}^{+\infty} e^{-2\alpha r}  {\dd}r\right)
\nonumber \\ \leq & \mathscr{I}_{\m}(o)\left(  \frac{\delta^3 \mathfrak{s}_{-k^2}^{n-1}(\delta)}{3} +   \frac{1}{2\alpha}  \right).\label{frsuprhp1}
\end{align}
On the other hand, since \eqref{sigmasmall} is  valid, we obtain
\begin{align*}
\int_M u_\alpha^2\dm&\geq \int_{S_oM} \int^\varepsilon_0 e^{-2\alpha r} \frac12 e^{-\tau(y)}r^{n-1}{\dd}\nu_o(y){\dd}r
\geq  \frac{\mathscr{I}_{\m}(o)}{2}\int^{\varepsilon}_0 e^{-2\alpha r}r^{n-1}{\dd}r \geq
\frac{\mathscr{I}_{\m}(o)}{2ne^{2\alpha\varepsilon}}\varepsilon^n,
\end{align*}
which together with \eqref{frsuprhp1} and \eqref{joulast} furnishes
\begin{align*}
\lim_{\alpha\rightarrow 0^+}\mathcal {H}_{o}(u_{\alpha})=\lim_{\alpha\rightarrow 0^+}\frac{\alpha^2  \ds\int_M u_\alpha^2 r^2\dm  }{ \ds\int_M u_\alpha^2 \dm  }\leq \lim_{\alpha\rightarrow 0^+} 2ne^{2\alpha\varepsilon}\varepsilon^{-n}\alpha^2\left(  \frac{\delta^3\mathfrak{s}_{-k^2}^{n-1}(\delta)}{3} +   \frac{1}{2\alpha}  \right)=0,
\end{align*}
 completing the proof.
\end{proof}


\begin{corollary}\label{failHPW1}
	Let $(M,F,\m)$ be  an $n$-dimensional forward complete $\FMMM$ with
	$
	\mathbf{Ric}\geq -(n-1) k^2$ and $  \mathbf{S} \geq (n-1){h},
$
	where  $k,h$ are two constants satisfying $h> k\geq 0$.
	Then for every point $o\in M$,
	the Heisenberg--Pauli--Weyl uncertainty principle fails, i.e., \eqref{failHPW} holds.
\end{corollary}

	We now discuss the  indispensability of our assumption $h> k\geq 0$ in Theorem \ref{ThmFalHPW} and Corollary \ref{failHPW1}.

\begin{remark}\rm
	(i) One cannot relax the assumption $h> k\geq 0$  to  $h=0$ (and  $k=0$). Indeed, according to  \cite[Theorem 1.2]{HKZ}, if a Cartan--Hadamard manifold $(M,F)$ endowed with a smooth positive measure $\m$ satisfies
	\begin{equation}\label{slambdacondition}
		\mathbf{S}\leq 0, \quad \lambda_F(M)<+\infty,
	\end{equation}
	then the following  Heisenberg--Pauli--Weyl uncertainty principle holds:
	\[
	\inf_{u\in C^\infty_0(M)\backslash\{0\}}\mathcal {H}_{o}(u)\geq \frac{1}{\lambda_F(M)^2}\frac{n^2}{4}.
	\]
	Note that
	there are large classes of Cartan--Hadamard manifolds satisfying \eqref{slambdacondition}; for example, every Minkowski space endowed with the Busemann--Hausdorff measure. Therefore, the positivity condition $h>0$ in Theorem \ref{ThmFalHPW} cannot be relaxed, which is in a perfect concordance also with the setting of  \cite[Theorem 1.3]{HKZ}. In fact, as we shall see in Section \ref{infactlue2}, the positivity of the $S$-curvature leads to the infinity of the reversibility.
	
	(ii) One cannot relax the assumption $h> k\geq 0$  to  $h=k>0$. To see this, let $(M,F,\m)$ be  an $n$-dimensional noncompact forward complete $\FMMM$. Assume that there is  $o\in M$ with
	\[
	\mathbf{K}(\nabla r)\leq -k^2,\quad \mathbf{S}(\nabla r)\leq (n-1)k,
	\]
	for some $k>0$, where $r(x):=d_F(o,x)$. Then we have the (non-sharp) Heisenberg--Pauli--Weyl uncertainty principle:
	\begin{equation}\label{hpw-nonsharp}
			\left( \int_M \max\{ F^{*2}(\pm {\dd}u) \}\dm   \right)\left( \int_M r^2 u^2\dm  \right)\geq \frac14 \left( \int_M u^2\dm \right)^2,\quad \forall\,u\in C^\infty_0(M).
	\end{equation}



\noindent Indeed, due to Wu--Xin \cite[Lemma 3.3, Theorem 4.1]{WX}, we have $\m$-a.e.\ in $M$ that
\[
\Delta r\geq  (n-1)k {\coth(kr)}-(n-1)k=   \frac{2(n-1)k}{e^{2kr}-1}   > 0.
\]
The divergence theorem, relation \eqref{dualff*}, the eikonal equation and H\"older's inequality yield
\begin{align*}
\left(2\int_M u^2\dm   \right)^2&\leq \left(2 \int_M (1+r\Delta_{\m} r)u^2\dm     \right)^2=\left( \int_M u^2 \Delta_{\m}(r^2)\dm\right)^2\\
&=16 \left( \int_M u r\langle \nabla r,{\dd}u\rangle\dm  \right)^2\leq 16\left(  \int_M |u| r \max\{F^*(\pm {\dd}u) \}\dm\right)^2\\
&\leq 16 \left( \int_M|u|^2 r^2\dm    \right)\left(\int_M \max\{F^{*2}(\pm {\dd}u) \}\dm  \right),
\end{align*}
which gives \eqref{hpw-nonsharp}.
\end{remark}

\begin{remark}\rm The interpretation of the classical  Heisenberg--Pauli--Weyl uncertainty principle in quantum mechanics is that the position and momentum of a	given particle cannot be accurately determined simultaneously; this phenomenon can be captured in the presence of the `gap' $n^2/4$ in relation \eqref{HPW-1} (or the value $1/4$ in the non-sharp inequality \eqref{hpw-nonsharp}). However, in our geometric context provided by Theorem \ref{ThmFalHPW}, the theoretical possibility of simultaneous detection of position and momentum of a particle still persist due to the gap-vanishing fact.
\end{remark}

%

\subsection{Failure of Caffarelli--Kohn--Nirenberg interpolation inequality}

\begin{theorem}\label{ThmFalCKN}
Let $(M,F,\m)$ be  an $n$-dimensional forward complete $\FMMM$ and let $k,h$ be nonnegative constants with $h> k\geq 0$.
Suppose that there exists a point $o\in M$  such that
\[
 \mathbf{Ric}(\nabla r)\geq -(n-1) k^2, \quad  \mathbf{S}(\nabla r)\geq (n-1){h},
\]
where $r(x):=d_F(o,x)$. Then the Caffarelli--Kohn--Nirenberg interpolation inequality fails, i.e.,
\begin{equation}\label{CKNfail}
\inf_{u\in C^\infty_0(M)\backslash\{0\}}\frac{\left(\ds\int_M F^{*2}({\dd}u) \dm   \right)\left(\ds\int_M  \frac{|u|^{2p-2}}{r^{2q-2}}\dm   \right)  }{\left(\ds\int_M  \frac{|u|^{p}}{r^{q}}\dm   \right)^2}=0,
\end{equation}
for any $p,q\in \mathbb{R}$ with
\[
0<q<2<p<+\infty, \quad 2<n<\frac{2(p-q)}{p-2}.
\]
\end{theorem}

\begin{proof} For every $\alpha>0$, let us consider the function
\[
u_\alpha(x):=-\left[ 1+(\alpha\, r(x))^{2-q}  \right]^{\frac{1}{2-p}}.
\]
We split the proof into three steps.

\smallskip

\textit{Step 1.} We first show that
\begin{equation}\label{fintienessthreelimi}
 \int_M F^{*2}({\dd}u_\alpha)\dm <+\infty, \quad  \int_M  \frac{|u_\alpha|^{2p-2}}{r^{2q-2}}\dm   <+\infty, \quad \int_M  \frac{|u_\alpha|^{p}}{r^{q}}\dm  <+\infty.
\end{equation}
Since $\lim_{r\rightarrow 0^+}\mathfrak{s}^{n-1}_{-k^2}(r)/r^{n-1}=1$, there exists an $\varepsilon=\varepsilon(n,k)\in (0,\mathfrak{i}_o)$ such that
\begin{equation}\label{estkh1}
r^{2-2q}\mathfrak{s}^{n-1}_{-k^2}(r)\leq 2r^{n-2q+1}, \quad \forall r\in (0,\varepsilon).
\end{equation}
On the other hand, since $\eta:=h-k>0$, a similar argument yields  $\delta=\delta(n,q,k,h)\in (\varepsilon,+\infty)$ with the property that
\begin{equation}\label{estkh2}
e^{-(n-1)hr}\mathfrak{s}_{-k^2}^{(n-1)}(r)r^{2-2q}\leq e^{-\frac{(n-1)\eta}2 r},\quad \forall \,r\in (\delta,+\infty).
\end{equation}
An argument similar to \eqref{st11ong} combined with \eqref{estkh1} and \eqref{estkh2}   implies that
\begin{align}
\int_M \frac{|u_\alpha|^{2p-2}}{r^{2q-2}} \dm =&\int_M \left[ 1+ (\alpha\, r)^{2-q}   \right]^{\frac{2p-2}{2-p}}r^{2-2q}\dm\notag\\
\leq&\left(\int_{S_oM}e^{-\tau(y)}{\dd}\nu_o(y)\right)\left(\int_0^{+\infty} \left[ 1+ (\alpha\, r)^{2-q}   \right]^{\frac{2p-2}{2-p}}r^{2-2q} e^{-(n-1)hr}\mathfrak{s}_{-k^2}^{n-1}(r){\dd}r\right)\notag\\
\leq & \mathscr{I}_{\m}(o) \int^{+\infty}_0 e^{-(n-1)hr}r^{2-2q}\mathfrak{s}_{-k^2}^{n-1}(r){\dd}r\nonumber\\=&\mathscr{I}_{\m}(o) \left( \int^{\varepsilon}_0+\int_\varepsilon^\delta+\int_\delta^{+\infty}  e^{-(n-1)hr}r^{2-2q}\mathfrak{s}_{-k^2}^{n-1}(r){\dd}r \right)\notag\\
\leq & \mathscr{I}_{\m}(o)\left(2 \int^{\varepsilon}_0r^{n-2q+1} {\dd}r+\int_\varepsilon^\delta    e^{-(n-1)hr}r^{2-2q}\mathfrak{s}_{-k^2}^{n-1}(r){\dd}r+\int_\delta^{+\infty}   e^{-\frac{(n-1)\eta}2 r} {\dd}r              \right)\notag\\
\leq & \mathscr{I}_{\m}(o)\left(   \frac{2\varepsilon^{n-2q+2}}{n-2q+2}+\int_\varepsilon^\delta    e^{-(n-1)hr}r^{2-2q}\mathfrak{s}_{-k^2}^{n-1}(r){\dd}r+\frac{2 e^{-\frac{(n-1)\eta \delta}{2}}}{(n-1)\eta} \right)\nonumber\\=:&C_1(n,q,k,h,\delta,\varepsilon)<+\infty.\label{cknes1}
\end{align}
By an analogous computation one can also show $\int_M  \frac{|u_\alpha|^{p}}{r^{q}}\dm  <+\infty$.
Moreover, a direct calculation yields
\[
{\dd}u_\alpha=\alpha^{2-q}\left(  \frac{2-q}{p-2}  \right)\left[ 1+(\alpha\, r)^{2-q}  \right]^{\frac{p-1}{2-p}} r^{1-q}{\dd}r,
\]
which together with \eqref{cknes1} implies
\begin{equation}\label{chkfnormes}
\int_M F^{*2}({\dd}u_\alpha)\dm=\alpha^{2(2-q)}\left(  \frac{2-q}{p-2}  \right)^2 \int_M\frac{|u_\alpha|^{2p-2}}{r^{2q-2}}\dm<+\infty.
\end{equation}
This completes the proof of \eqref{fintienessthreelimi}.

\smallskip

\textit{Step 2.}  We show that there exists a sequence $u_{\alpha,l}\in \Lip_0(M)$ such that
\begin{equation}\label{uapunli}
\mathcal {G}_{o,p,q}(u_{\alpha})=\lim_{l\rightarrow \infty}\mathcal {G}_{o,p,q}(u_{\alpha,l}),
\end{equation}
where
\begin{equation}\label{cknfunctional}
\mathcal {G}_{o,p,q}(u):=\frac{\left(\ds\int_M F^{*2}({\dd}u) \dm   \right)\left(\ds\int_M  \frac{|u|^{2p-2}}{r^{2q-2}}\dm   \right)  }{\left(\ds\int_M  \frac{|u|^{p}}{r^{q}}\dm   \right)^2}.
\end{equation}
As before, it is enough to consider the case when $\sup_{x\in M}r(x)=+\infty$;  otherwise, $M$ is  compact and the sequence would be  $u_{\alpha,n}:=u_\alpha\in \Lip_0(M)$.
Then for small $\varepsilon\in (0,1)$, set
\[
u_{\alpha, \varepsilon}:=
\begin{cases}
-[ 1+(\alpha\, \varepsilon)^{2-q}  ]^{\frac{1}{2-p}}, &\text{if }  r < \varepsilon,\\
\\
-[ 1+(\alpha\, r)^{2-q}  ]^{\frac{1}{2-p}}, &\text{if }   r \geq \varepsilon,
\end{cases}
\qquad v_{\alpha,\varepsilon}:=\min\{u_{\alpha, \varepsilon}+\varepsilon,0 \}, \qquad u_{\alpha,l}:=v_{\alpha,1/l}, \ l\in \mathbb{N}.
\]
On can easily show that  $v_{\alpha,\varepsilon},u_{\alpha,l}\in \Lip_0(M)$ since $u_\alpha\in C^\infty(M)$ with $\lim_{r(x)\rightarrow +\infty}u_\alpha(x)=0$. In addition, the construction shows the  pointwise convergences
\[
|u_{\alpha,l}|\nearrow |u_\alpha|,\quad F^*({\dd}u_{\alpha,l})\nearrow F^*({\dd}u_\alpha)\ \ {\rm as} \ \ l\rightarrow \infty.
\]
Then \eqref{uapunli} can be derived by these limits,  \eqref{fintienessthreelimi} and the dominated convergence theorem.

\smallskip

\textit{Step 3.}
On the one hand,  by the same method as employed in the proof of Lemma \ref{Hardyfinallemma} one can show
\begin{equation}\label{CKNfail1111}
\inf_{u\in C^\infty_0(M)\backslash\{0\}} \mathcal {G}_{o,p,q}(u)\leq \liminf_{\alpha \rightarrow 0^+} \mathcal {G}_{o,p,q}(u_{\alpha}).
\end{equation}
On the other hand, since
\eqref{sigmasmall} remains valid,
we have
\begin{align}
\int_M  \frac{|u_\alpha|^{p}}{r^{q}}\dm=&\int_M \left[ 1+ (\alpha\, r)^{2-q}   \right]^{\frac{p}{2-p}}r^{-q}\dm\notag\\
\geq&   \left(\int_{S_oM}e^{-\tau(y)}{\dd}\nu_o(y)\right)\left( \int^{\varepsilon}_0   \left[ 1+ (\alpha\, r)^{2-q}   \right]^{\frac{p}{2-p}}r^{-q}\frac12 r^{n-1}{\dd}r    \right)\notag\\
\geq &\frac{\mathscr{I}_{\m}(o)}2 \left[ 1+ (\alpha\, \varepsilon)^{2-q}   \right]^{\frac{p}{2-p}}\int^{\varepsilon}_0  r^{n-1-q}{\dd}r\notag\\
=&\frac{\mathscr{I}_{\m}(o)\varepsilon^{n-q}}{2(n-q)}\left[ 1+ (\alpha\, \varepsilon)^{2-q}   \right]^{\frac{p}{2-p}}=:C_2({\alpha},n,p,q,k,h,\varepsilon)>0.\label{CKNest2}
\end{align}
Now it follows from   \eqref{CKNfail1111}, \eqref{cknes1}, \eqref{chkfnormes}   and \eqref{CKNest2} that
\begin{align*}
\inf_{u\in C^\infty_0(M)}\mathcal {G}_{o,p,q}(u)\leq& \liminf_{\alpha \rightarrow 0^+} \mathcal {G}_{o,p,q}(u_\alpha)\\
=&\liminf_{\alpha \rightarrow 0^+}\alpha^{2(2-q)}\left(  \frac{2-q}{p-2}  \right)^2  \left[\frac{\ds \int_M\frac{|u_\alpha|^{2p-2}}{r^{2q-2}}\dm}{\ds\int_M  \frac{|u_\alpha|^{p}}{r^{q}}\dm }\right]^2\\
\leq &\liminf_{\alpha \rightarrow 0^+}\alpha^{2(2-q)}\left(  \frac{2-q}{p-2}  \right)^2 \left(\frac{C_1}{C_2}\right)^2= 0,
\end{align*}
i.e., \eqref{CKNfail} follows.
\end{proof}

\begin{remark}
We note that if $h=k>0$ and $q\in (1,2)$, then \eqref{estkh2} and $\int_M  \frac{|u_\alpha|^{p}}{r^{q}}\dm  <+\infty$ still hold and hence, the  Caffarelli--Kohn--Nirenberg interpolation inequality fails. However, it remains open the question when $h=k>0$ and $q\in (0,1)$, as we cannot guarantee the validity of \eqref{estkh2}.
\end{remark}

\begin{corollary}\label{failCKN1}
Let $(M,F,\m)$ be  an $n$-dimensional forward complete $\FMMM$ with
$
\mathbf{Ric}\geq -(n-1) k^2$ and $  \mathbf{S} \geq (n-1){h},
$
where  $k,h$ are two constants satisfying $h> k\geq 0$.
Then for every point $o\in M$,
 the Caffarelli--Kohn--Nirenberg interpolation inequality fails, i.e., \eqref{CKNfail} holds for any $p,q\in \mathbb{R}$ with
 \[
 0<q<2<p<+\infty,\quad 2<n<\frac{2(p-q)}{p-2}.
 \]
\end{corollary}

\begin{proof}[Proof of Theorem \ref{mainThm1}] The statements are direct consequences of Corollaries \ref{basicwanisheigven}, \ref{failhardcor1}, \ref{failHPW1} and \ref{failCKN1}, respectively.
\end{proof}

\section{Influence of $S$-curvature on geometric aspects: proof of Theorems \ref{reversibinfite}\,\&\ref{Spositivenoncom}}\label{infactlue2}

\subsection{Infinite reversibility}

In this subsection we prove Theorem  \ref{reversibinfite}, which establish a surprising relationship between the $S$-curvature and reversibility. To do this, we recall the following Hardy inequality established by Zhao  \cite[Theorem 3.8]{ZhaoHardy}.
\begin{theorem} {\rm (Zhao \cite{ZhaoHardy})}\label{hardytimethem}
Let $(M,F,\m)$ be an $n$-dimensional  forward complete noncompact $\FMMM$ with
\[
\mathbf{K}\leq 0, \quad \mathbf{S}\geq 0.
\]
Given a fixed point $o\in M$, set ${\mathop{r}\limits^{\leftarrow}}(x):=d_F(x,o)$.
Then for any $p\in (1,n)$ and $\beta\in (-n,\infty)$,
\[
\int_M \max \{ F^{*p}(\pm {\dd}u) \}{\mathop{r}\limits^{\leftarrow}}^{\beta+p}\dm \geq \left(  \frac{n+\beta}{p} \right)^p \int_M |u|^p {\mathop{r}\limits^{\leftarrow}}^{\beta}\dm,\quad \forall\,u\in C^\infty_0(M).
\]
In particular, the constant $\left(  \frac{n+\beta}{p} \right)^p$ is sharp if $\lambda_F(M)=1$.
\end{theorem}



\begin{proof}[Proof of Theorem \ref{reversibinfite}]

Suppose by contradiction that $\lambda_F(M)$ is finite.
Choosing an arbitrary point $o\in M$, set $r(x)=d_F(o,x)$ and ${\mathop{r}\limits^{\leftarrow}}(x):=d_F(x,o)$. Then
\eqref{non-symmetric}  and \eqref{revercondin} yield
\begin{equation}\label{lambdaestima}
{\lambda^{-1}_F(M)}\,r\leq {\mathop{r}\limits^{\leftarrow}} \leq \lambda_F(M)\,  r,   \quad \max\{ F^{*}(\pm {\dd} u ) \}\leq \lambda_F(M)\, F^{*}({\dd} u ).
 \end{equation}
 For any $p\in (1,n)$,
 Theorem \ref{hardytimethem} (by choosing $\beta=-p$) combined with \eqref{lambdaestima} furnishes
\begin{equation*}
\lambda^p_F(M) \int_M   F^{*p}({\dd} u ) \dm \geq \left( \frac{n - p}{p\lambda_F(M)} \right)^p \int_M \frac{|u|^p}{r^p} \dm,
\quad \forall\, u \in C^{\infty}_0(M),
\end{equation*}
which implies
\begin{equation*}\label{ZhaoHardy_betaspecial}
\inf_{u\in C^\infty_0(M)\setminus \{0\}}\frac{\ds\int_M   F^{*p}({\dd} u ) \dm}{\ds\int_M \frac{|u|^p}{r^p} \dm } \geq \left( \frac{n - p}{p\lambda^2_F(M)} \right)^p>0.
\end{equation*}
However, according to Corollary \ref{failhardcor1}, the left hand side of the above estimate is zero, which is a contradiction. Therefore, we necessarily have $\lambda_F(M) = + \infty$, which concludes the proof.
\end{proof}

\begin{corollary}
 Let $(M,F)$ be an $n$-dimensional Cartan--Hadamard manifold equipped with a smooth positive measure $\m$. Assume that
$\mathbf{K}\geq -k^2$ and $\mathbf{S} \geq (n-1){h}$ for some constants
$k,h$ with $h\geq k\geq 0$ and $h>0$.
Then  $\lambda_F(M)=+\infty$.
\end{corollary}

\subsection{Topological properties of projectively flat manifolds}\label{projceflatmanifold} In this subsection we prove Theorem  \ref{Spositivenoncom}, which provides another surprising phenomenon coming from the behavior of the $S$-curvature.
Throughout this subsection, let $\Omega$ be a non-empty domain (i.e., connected open subset) in $\mathbb{R}^n$ with  $\mathbf{0}\in \Omega$.
A Finsler metric $F$ on $\Omega$ is
 said to be \textit{projectively flat} if the images of all geodesics are straight lines. Analytically, this property can be characterized as follows.
\begin{lemma} {\rm (Shen \cite{ShenSpray})} \label{projeclemmash}
 The metric $F$ on $\Omega$ is projectively flat if and only if  the geodesic coefficients defined by \eqref{goedcooff} satisfy  $G^i = P y^i$,
where
\begin{equation}\label{prodef}
P :=  \frac{y^k}{2F}\frac{\partial F }{\partial {x^k}}
\end{equation}
 is called the {\it projective factor} of $F$.
 \end{lemma}
 \begin{remark}\label{prjecremark}
 The projective factor $P$ is positively $1$-homogenous, i.e., $P(x,\alpha y)=\alpha P(x,y)$ for any $\alpha>0$. Then Euler's theorem (cf.\,\cite[Theorem 1.2.1]{BCS}) yields
 \begin{equation}\label{pohom}
 y^m\frac{\partial P}{\partial y^m}(x,y)=P(x,y).
 \end{equation}
 \end{remark}

In the following result we prove that the completeness of a projectively flat Finsler manifold is deeply influenced by  the $S$-curvature.
\begin{theorem} \label{SpositiveBackimpossible}
Let $(\Omega,F,\mathscr{L}^n)$ be an $n$-dimensional projectively flat Finsler manifold  endowed with the Lebesgue measure. Assume that  $\Omega$ is   bounded in $\mathbb{R}^n$.
Then $(\Omega,F)$ cannot be backward $($resp., forward$)$ complete whenever $\mathbf{S} \geq 0$ $($resp., $\mathbf{S}\leq 0$$).$
\end{theorem}
\begin{proof} We only consider $\mathbf{S} \geq 0$; the other case can be treated similarly. Let $(x^i)$ be the standard coordinate system of $\mathbb{R}^n$.
It follows by \eqref{Scurvature},  Lemma \ref{projeclemmash} and \eqref{pohom} that
 \[
0\leq \mathbf{S}(x,y) =  \frac{\partial G^m(x,y) }{\partial y^m} =\frac{\partial (P(x,y) y^m)}{\partial y^m} = (n+1) {P(x,y)},
\]
which implies
\begin{equation}\label{pnonnenga}
P(x,y) \geq 0, \quad \forall\,(x,y)\in T\Omega\backslash\{0\}.
\end{equation}

As $\mathbf{0}\in \Omega$ and $F$ is projectively flat,
a geodesic $t\mapsto \gamma(t)$  starting from a point $\mathbf{x}_o  \in \Omega$ in the direction $\mathbf{a}  \in \partial \Omega$ has the form
\[
\gamma(t) = \mathbf{x}_o+f(t)\mathbf{a},\quad t\in (-\varepsilon,\varepsilon),
\]
with $f(0) = 0$ and $f'(t) > 0$. The existence of $\varepsilon>0$ follows from the standard theory of ODE, as $\gamma$ satisfies the geodesic equation \eqref{geodesequ}.

Assume by contradiction that $(\Omega,F,\m)$ is  backward complete; then $\gamma$ can be extended to $t\in (-\infty,\varepsilon)$, which implies $f\in C^2((-\infty,\varepsilon))$ (in fact, $f$ is smooth).  Since  $\gamma$ satisfies \eqref{geodesequ}, it follows by \cite[Exercise 5.2.4/(d)]{BCS} that  $\gamma$ has  constant velocity, i.e.,
\begin{equation}\label{constvelocity}
F(\gamma(t),\gamma'(t)) =F(\gamma(t),f'(t)\mathbf{a})=\text{const}=:c,
\end{equation}
 which means
$f'(t)>0$ for all $t\in (-\infty,\varepsilon)$.
Due to Lemma \ref{projeclemmash},  the   geodesic equation \eqref{geodesequ} reduces to
\begin{equation}\label{Backf''}
  f''(t)  + 2 P \left(\gamma(t), f'(t)\mathbf{a}\right) f'(t)   = 0,
\end{equation}
which together with \eqref{pnonnenga}, Remark \ref{prjecremark} and $f'(t)>0$ yields
\[
 f''(t) = - 2 \left(f' (t)\right)^2 P (\gamma(t), \mathbf{a}) \leq 0.
 \]
Therefore, by a Taylor expansion, it follows that $f$ satisfies
\begin{equation} \label{fconvex}
f(t)\leq f(0) + f'(0) (t -0) = f'(0) t, \quad \forall\,t\in (-\infty,\varepsilon).
\end{equation}
Note that the image of $\gamma(t)=\mathbf{x}_o+f(t)\mathbf{a}$ is the straight line $l(s) = \mathbf{x}_o + s \mathbf{a}$. Since $\Omega$ is bounded, we may assume that
 $l(s) = \mathbf{x}_o + s \mathbf{a}  $ intersects $\partial \Omega$  in the opposite direction of $\mathbf{a}$ at some $s_-\in (-\infty,0)$. Then the Euclidean distance satisfies
 \[
 |f(t)\mathbf{a}|=|\gamma(t)-\mathbf{x}_o| \leq |l(s_-)-\mathbf{x}_o |=|s_-\mathbf{a}|,\quad \forall\,t\in (-\infty,0),
 \]
 which combined with  \eqref{fconvex} implies
\[
-\infty < s_- \leq f(t) \leq  f'(0) t,  \quad \forall\,t\in (-\infty,0).
\]
Since $f'(0)>0$, the latter estimate provides a contradiction. Then
$(\Omega,F,\m)$ is not backward complete.
\end{proof}

Theorem \ref{SpositiveBackimpossible} is optimal in the sense that the boundedness of $\Omega$ in $\mathbb{R}^n$ cannot be canceled. Indeed, the following two examples point out that  in both cases (i.e., $\mathbf{S} =0$ and $\mathbf{S}>0$), if $\Omega$ is unbounded in $\mathbb{R}^n$, the corresponding  projectively flat manifold can be  complete, i.e., both forward and backward complete.
\begin{example}
Let $(\Omega,F,\mathscr{L}^n):=(\mathbb{R}^n, |\cdot|, \mathscr{L}^n)$ be the Euclidean space endowed with the Lebesgue measure. Obviously, $\Omega$ is unbounded in $\mathbb{R}^n$;
although $(\Omega,F,\mathscr{L}^n)$ is projectively flat with $\mathbf{S} =0$,  it  is complete.
\end{example}

\begin{example}\label{Ex_S2}
Let $(\Omega,F,\mathscr{L}^n):=(\mathbb{R}^n, F, \mathscr{L}^n)$, where
\begin{equation} \label{SpositiveEx2}
F(x,y): = |y| + \frac{\langle x,y \rangle}{2 \sqrt{1+ |x|^2}}.
\end{equation}
According to Lemma \ref{projeclemmash}, a direct computation shows the projective flatness of $F$ and
\[
 P = \frac{1}{2 F} \frac{\partial F}{\partial x^k} y^k = \frac{(1+ |x|^2)|y|^2-\langle x,y \rangle^2}{4 (1+ |x|^2)^{\frac{3}{2}} \Big( |y| + \frac{\langle x,y \rangle}{2\sqrt{1+ |x|^2}} \Big) } > 0,\quad
 \mathbf{S} = {(n+1)P} >0.
 \]
As $F$ is projectively flat, the minimal geodesic from  $\mathbf{0}$ to $x$ has the same image of the straight line  $\gamma(t) = t x$, $t \in [0,1]$.
Then the distances ${d}_{F}(\mathbf{0},x)$ and ${d}_{F}(x,\mathbf{0})$ can be obtained by
\begin{align*}
{d}_{F}(\mathbf{0},x)& = \int^1_0 F(tx, x) {\dd}t = \int^1_0 |x| + \frac{t |x|^2}{2\sqrt{1+ t^2 |x|^2}} {\dd}t  = |x| -\frac{1}{2} + \frac{1}{2}\sqrt{1+  |x|^2},
\\
{d}_{F}(x,\mathbf{0}) &= \int^1_0 F(tx, - x) {\dd}t = \int^1_0 |x| - \frac{t |x|^2}{2\sqrt{1+ t^2 |x|^2}} {\dd}t  = |x| + \frac{1}{2} - \frac{1}{2}\sqrt{1+ |x|^2}.
\end{align*}
Given $x_1,x_2\in \mathbb{R}^n$, the triangle inequality implies
\begin{equation}\label{ExampleS>0_dx1x2}
d_F(x_1,x_2)\geq d_F(\mathbf{0},x_2)- d_F(\mathbf{0},x_1)=|x_2| + \frac{1}{2}\sqrt{1+  |x_2|^2} - |x_1| - \frac{1}{2}\sqrt{1+  |x_1|^2},
\end{equation}
\begin{equation}\label{ExampleS>0_dx1x2_inverse}
d_F(x_1,x_2)\geq d_F(x_1, \mathbf{0})- d_F(x_2, \mathbf{0})=|x_1| - \frac{1}{2}\sqrt{1+  |x_1|^2} -|x_2| + \frac{1}{2}\sqrt{1+  |x_2|^2} .
\end{equation}

Let $c$ be the unit-speed geodesic from $x_1$ to $x_2$. By the projective flatness,  $c$ can be extended beyond $x_2$.
In fact, $|c(t)|\rightarrow +\infty$ as $t\to \infty,$ since the image of $t\mapsto c(t)$ is a straight line and $\Omega=\mathbb{R}^n$.
Then \eqref{ExampleS>0_dx1x2} yields
\[
d_F(c(0),c(t))=d_F(x_1,c(t)) \geq |c(t)| + \frac{1}{2}\sqrt{1+  |c(t)|^2} - |x_1| - \frac{1}{2}\sqrt{1+  |x_1|^2}  \rightarrow +\infty,
\]
as $ |c(t)| \rightarrow +\infty$, which means that every unit-speed geodesic can be extended forward to $t=+\infty$. Hence, $(\mathbb{R}^n, F, \mathscr{L}^n)$ is forward complete.

In a similar manner, $c$ can be also extended before $x_1$. By  \eqref{ExampleS>0_dx1x2_inverse}, we have
\[
d_F(c(t),c(0))=d_F(c(t),x_1) \geq |c(t)| -\frac{1}{2}\sqrt{1+  |c(t)|^2} - |x_1| + \frac{1}{2}\sqrt{1+  |x_1|^2}  \rightarrow +\infty,
\]
as $ |c(t)| \rightarrow +\infty$, which means that every unit-speed geodesic can be extended backward to $t=- \infty$. Then $(\mathbb{R}^n, F, \mathscr{L}^n)$ is backward complete.
In conclusion, $(\mathbb{R}^n, F, \mathscr{L}^n)$ is complete.

Moreover,   a direct computation shows that
\[
\lambda_F(\mathbb{R}^n) =\sup_{x \in \mathbb{R}^n} \sup_{y \in T_x\mathbb{R}^n \setminus \{ 0 \}} \frac{F(x, -y)}{F(x,y)} = \sup_{x \in \mathbb{R}^n}\sup_{y \in \mathbb{R}^n \setminus \{ 0 \}} \frac{|y| - \frac{\langle x,y \rangle}{2 \sqrt{1+ |x|^2}}}{|y| + \frac{\langle x,y \rangle}{2 \sqrt{1+ |x|^2}}} = 3.
\]
\end{example}

A direct consequence of Theorem \ref{SpositiveBackimpossible} reads as follows:
\begin{corollary}\label{cor_S}
Let  $(\Omega,F,\mathscr{L}^n)$ be an $n$-dimensional forward complete   projectively flat Finsler manifold  endowed with the Lebesgue measure.
Assume that $\Omega$ is  bounded  in $\mathbb{R}^n$ and  $\mathbf{S}\geq  0$. Then  $\lambda_F(M)=+\infty$.
\end{corollary}

Having Theorem \ref{SpositiveBackimpossible}, we now have:

\begin{proof}[Proof of Theorem \ref{Spositivenoncom}] The proof is a modification of that of Theorem \ref{SpositiveBackimpossible}. Hence, we just give a sketch.  Since $(\Omega,F)$ is forward complete, a geodesic $\gamma$  starting from $\mathbf{0}\in \Omega$ in the direction $\mathbf{a} \in \mathbb{S}^{n-1}$ has the form
\begin{equation}\label{geodesicque}
\gamma(t) = f(t)\mathbf{a},\quad t\in (-\varepsilon,+\infty),
\end{equation}
with $f(0) = 0$ and $f'(t) > 0$.
The assumption $\mathbf{S}\geq (n-1)h>0$ implies
\[
(n+1)P(x,y)=\mathbf{S}(x,y)\geq (n-1)h F(x,y),\quad \forall\,(x,y)\in T\Omega\backslash\{0\},
\]
which together with \eqref{constvelocity} and  \eqref{Backf''} yields
\[
 f''(t) =-2 P \left(\gamma(t), \gamma'(t)\right) f'(t)\leq -  \frac{2(n-1)}{n+1} h c f'(t),
 \]
 where $c=F(\gamma(t),\gamma'(t))$ is the constant speed of $\gamma$.
Since $f'(t)>0$,  the above inequality can be written into
\[ [\ln f'(t)]' \leq -  \frac{2(n-1)}{n+1} h c, \]
which yields
\[
 f'(t) \leq e^{-  \frac{2(n-1)}{n+1} h c t}  f'(0).
 \]
Integrating the above inequality again from $0$ to $t$  yields
\[
f(t) \leq - \frac{(n+1) f'(0)}{2(n-1) h c}   e^{-  \frac{2(n-1)}{n+1} h c t} + \frac{(n+1) f'(0)}{2(n-1) h c}.
\]
Then we have
 \[
 f(t)\leq \frac{(n+1) f'(0)}{2(n-1) h c}=:R<+\infty, \quad \forall\, t\in [0,+\infty).
 \]
This inequality together with \eqref{geodesicque}, the forward completeness, and the projective flatness of $(\Omega,F)$ implies
\[
\Omega\subset \mathbb{B}^n(R):=\{y\in \mathbb{R}^n\,:\, |y|\leq R\},
\]
 i.e., $\Omega$ is bounded in $\mathbb{R}^n$.
 By  Theorem \ref{SpositiveBackimpossible}, $(\Omega,F)$ is not backward complete.

 Moreover, we have $\lambda_F(\Omega)=+\infty$; indeed, if $\lambda_F(\Omega)$ is finite,  $(\Omega,F)$ would be also backward complete, coming from its forward completeness.
\end{proof}


\section{Examples \& Counterexamples:  Funk metric spaces}\label{sectionFunk}
The purpose of this section is twofold. First, in subsection \ref{section-funk} we investigate Funk metric spaces,  a large class of non-Riemannian Finsler manifolds that provide the geometric background for the applicability of our main results, Theorems \ref{mainThm1}-\ref{Spositivenoncom}.  Second, in subsection \ref{Counterexamples}, as we already noticed in the Introduction, we show that the combined assumption on the Ricci and $S$-curvatures cannot be replaced by the weighted Ricci curvature ${\bf Ric}_N$, in spite of the fact that the latter notion is a combination of $\mathbf{Ric}$ and $\mathbf{S}$.


Throughout this section,
let $\Omega \subset \mathbb{R}^n$ be a bounded strongly convex domain with $\mathbf{0}\in \Omega$, let $(x^i)$ denote the standard coordinate system of $\mathbb{R}^n$ and let $\vol(\cdot)$ denote the Euclidean volume.

\subsection{Funk metric spaces: applicability of Theorems \ref{mainThm1}-\ref{Spositivenoncom}}\label{section-funk}

\begin{definition}\label{funkdef}
A Finsler metric $F$ on $\Omega$ is called a ({\it general}) {\it Funk metric} if
\[
x + \frac{y}{F(x,y)}  \in \partial \Omega.
\]
for  every $x\in \Omega$ and  $ y \in T_x \Omega\backslash\{0\}$. In the sequel, the pair $(\Omega,F)$ is called a ({\it general}) {\it Funk metric space}.
\end{definition}

The main result of this subsection can be formulated as follows:
\begin{theorem}\label{theorem-final}
Any $n$-dimensional Funk metric space  $(\Omega,F,\m_{BH})$ endowed with the Busemann--Hausdorff measure verifies the assumptions of Theorems \ref{mainThm1}-\ref{Spositivenoncom}. In particular, the Hardy  inequality, Heisenberg--Pauli--Weyl uncertainty principle and  Caffarelli--Kohn--Nirenberg inequality fail on $(\Omega,F,\m_{BH})$.
\end{theorem}

Before giving the proof of Theorem \ref{theorem-final}, we recall/collect some useful properties of  Funk metric spaces.
We first notice that  every Funk metric is projectively flat.
In addition, according to Shen \cite{Sh1} and Li \cite{Li}, a  Finsler metric $F$ on $\Omega$ is a Funk metric if and only if
there is a unique Minkowski norm $\phi = \phi(y)$ on $\mathbb{R}^n$ such that
\begin{equation} \label{phiF}
	\partial\Omega=\phi^{-1}(1),\quad F(x,y) = \phi(y + x F(x,y)).
\end{equation}
In particular,
\begin{equation}\label{omegedom}
	\Omega=\{x\in \mathbb{R}^n\,:\, \phi(x)<1\}.
\end{equation}
Having this characterization of Funk metrics, we state the following set of properties; since we did not find them in the literature in this general form, we provide also their proof.

\begin{proposition} \label{Funkdist-0}
	Let $(\Omega,F)$  be a Funk metric space and let $\phi = \phi(y)$ the Minkowski norm satisfying \eqref{phiF}. Then the following statements hold$:$
	\begin{enumerate}[\rm (i)]
		\item\label{distexprseeion1} the distance function is given by
		\begin{equation}\label{phi_dpq}
			d_F(x_1, x_2) = \ln \left[\frac{\phi(z - x_1)}{\phi(z - x_2)}\right],\quad \forall\,x_1,x_2\in \Omega,
		\end{equation}
		where $z$ is the intersection point of the ray $l_{x_1 x_2}(t):= x_1 + t (x_2 - x_1)$, $t\geq 0$, with $\partial \Omega$, i.e.,
		\[
		z=x_1 + \lambda (x_2-x_1) \in \partial\Omega \quad \text{ for some }\lambda>1;
		\]

		\item\label{distexprseeion2} for any $x\in \Omega$, the distances to and from the origin are
		\begin{equation}\label{Funkd0xdx0}
			d_F (\mathbf{0},x) = -\ln (1-\phi(x)), \qquad   d_F(x,\mathbf{0}) =  \ln (1+ \phi(-x));
		\end{equation}

		\item\label{distexprseeion3}  $(\Omega,F)$ is forward complete but not backward complete with $\lambda_F(\Omega)=+\infty$.
		Furthermore,
		\begin{equation}\label{Funkreversibiliey}
			\lambda_{F}(x) = \sup_{y \in T_x \Omega \setminus \{ 0 \}} \frac{F(x, -y)}{F(x,y)} \in \left[\frac{1 +\phi(x)}{1-\phi(x)}, \frac{\lambda_\phi +\phi(x)}{1-\phi(x)}\right],\quad \forall\,x\in \Omega.
		\end{equation}
	\end{enumerate}
\end{proposition}

\begin{proof}
	\eqref{distexprseeion1} Given $x_1,x_2\in \Omega$, let  $t\mapsto c(t)$ be a geodesic from $x_1$ to $x_2$ with $c(0)=x$ and $c'(0)=y$.
	Due to the projective  flatness of $(\Omega,F)$,
	for an appropriate parameter $t = t(s)$, we may assume that
	$
	c(s):= x + s y.
	$
	
	Since $x + \frac{1}{F(x,y)}y \in \partial \Omega$ by Definition \ref{funkdef}, we have $x+s y\notin \partial \Omega$ if $s>\frac{1}{F(x,y)}$, which implies $c(s)\in \Omega$ for $s\in [0,\frac{1}{F(x,y)}]$.  
	Then along the direction $y$,
	the corresponding point in $\partial \Omega$ is given by
	\[
	\partial\Omega \ni z:=x + \frac{y}{F(x, y)} = c(s) + \frac{y}{F(c(s), y)} = x+ s y + \frac{y}{F(c(s), c'(s))},
	\]
	which implies
	\[ F(c(s), c'(s)) = \frac{F(x, y)}{1- F(x, y)s} = \frac{F(y)}{1- F(y)s}, \]
	where $F(y)$ is $F(x,y)$ for short.
	Then formula (\ref{phi_dpq}) giving the distance ${d}_F(x_1, x_2)$  between $x_1 = c(s_1)$ to $x_2 = c(s_2)$ is
	\begin{align*}
		{d}_F(x_1, x_2)
		=& \int_{s_1}^{s_2} F(c(s), c'(s)) {\dd}s = \int_{s_1}^{s_2} \frac{F(y)}{1- F( y)s} {\dd}s = \ln \left[ \frac{1- F(y)s_1}{ 1-F(y)s_2} \right]= \ln  \left[ \frac{\frac{1}{F(y)}- s_1}{ \frac{1}{F(y)}-s_2} \right] \\
		=&\ln  \left[ \frac{(\frac{1}{F(y)}- s_1)\phi(y)}{ (\frac{1}{F(y)}-s_2)\phi(y)} \right] = \ln  \left[  \frac{\phi(\frac{y}{F(y)}- s_1 y )}{ \phi(\frac{y}{F(y)}-s_2 y)} \right] = \ln  \left[  \frac{\phi(x+\frac{y}{F(y)}-x- s_1 y)}{ \phi(x+\frac{y}{F(y)}-x-s_2 y)} \right] \\=& \ln \left[  \frac{\phi(z-x_1)}{ \phi(z-x_2)} \right].
	\end{align*}

	\smallskip
	
	\eqref{distexprseeion2} Given $x_1,x_2\in \Omega$,
	let $z := x_1 + \lambda (x_2-x_1) \in \partial\Omega$ for some  $\lambda > 1$. Then
	\eqref{phiF} and (\ref{phi_dpq}) yield
	\begin{align}
		1&=\phi(z) = \phi(x_1 + \lambda (x_2-x_1)),\label{boudary1}\\
		d_F(x_1, x_2)& = \ln \left[\frac{\phi(\lambda (x_2-x_1))}{\phi\left((\lambda -1)(x_2-x_1)\right)}\right]=\ln \frac{\lambda}{\lambda -1}.\label{lambda-dpq}
	\end{align}
	By choosing   $x_1 = \mathbf{0}$ and $x_2 = x$, we derive $\lambda =1/\phi(x)$ from \eqref{boudary1},
	which together with (\ref{lambda-dpq}) yields
	\[
	d_F(\mathbf{0},x) = \ln \frac{\frac{1}{\phi(x)}}{\frac{1}{\phi(x)} - 1} = - \ln (1- \phi(x)).
	\]
	On the other hand,
	when  $x_1 = x$ and $x_2 = \mathbf{0}$, from \eqref{boudary1} we have
	$\lambda = \frac{1}{\phi(-x)} +1$; then (\ref{lambda-dpq}) gives
	\begin{equation}\label{Funkdx0}
		d_F(x, \mathbf{0}) = \ln (1+ \phi(-x)).
	\end{equation}
	
	\smallskip
	
	\eqref{distexprseeion3}
	Given $x_1,x_2\in \Omega$, the triangle inequality combined with \eqref{Funkd0xdx0} yields
	\begin{equation}\label{dxx}
		d_F(x_1,x_2)\geq d_F(\mathbf{0},x_2)- d_F(\mathbf{0},x_1)=-\ln(1-\phi(x_2))+\ln(1-\phi(x_1)).
	\end{equation}
	Let $t\mapsto c(t)$, $t\in [0,1],$ be the unit-speed geodesic from $c(0)$ to $c(1)$ (i.e., the ray  $l_{c(0)c(1)}$). We extend $c$ to the point $z=c(0)+ \lambda (c(1)-c(0))\in \partial\Omega$; hence, \eqref{phiF}$_1$ implies
	\[
	\lim_{c(t)\rightarrow z}\phi(c(t))=\phi(z)=1,
	\]
	which combined with \eqref{dxx} yields
	\[
	t=d_F(c(0),c(t))=d_F(x_1,c(t))\geq -\ln(1-\phi(c(t)))+\ln(1-\phi(x_1))\rightarrow +\infty,
	\]
	as $\phi(c(t))\rightarrow 1^-$, which means that every unit-speed geodesic can be forward extended to $t=+\infty$. Hence, $(\Omega,F)$ is forward complete.  On the other hand, \eqref{Funkdx0} implies $d_F(x, \mathbf{0}) \rightarrow \ln 2$ as $\phi(-x)\rightarrow 1^-$. Hence, ($\Omega$, $F$) cannot be backward complete because the  bounded closed backward ball $\overline{B^{-}_{\mathbf{0}}(\ln 2)} = \Omega$ is not compact.

	Given $x\in \Omega$ and considering it as a (tangent) vector in $\mathbb{R}^n$, by     \eqref{dualff*}  we have
	\begin{equation}\label{xxirage}
		\langle x, \xi\rangle\leq \phi(x) \phi^*(\xi),\quad \forall\,\xi\in T^*_x\Omega=\mathbb{R}^n,
	\end{equation}
	with equality if and only if $\xi=\alpha \mathfrak{L}_\phi(x)$ for some $\alpha\geq 0$; here,
	$\mathfrak{L}_\phi$ stands for the Legendre transformation induced by $\phi$.  It costs no generality to assume that $\phi^*(\mathfrak{L}_\phi(x))\leq \phi^*(-\mathfrak{L}_\phi(x))$.
According to Shen \cite{Sh0} and Huang--Mo \cite{HuangMo}, the co-metric $F^{*}$ of $F$ is given by
\[	F^{*}(x,\xi) = \phi^{*}(\xi) - \langle x, \xi \rangle,\]
where $\phi^{*}$ is the co-metric of $\phi$. Hence, using \eqref{revercondin}, we obtain
\begin{align*}
		\lambda_{F}(x)  =&  \lambda_{F^*}(x)=\sup_{\xi \in T^*_x \Omega \setminus \{ 0 \}}\frac{F^*(x, -\xi)}{F^*(x,\xi)} = \sup_{\xi \in T^*_x \Omega \setminus \{ 0 \}}\frac{\phi^{*}(-\xi) + \langle x, \xi\rangle}{\phi^{*}(\xi) - \langle x, \xi\rangle}\\
		\geq  &  \frac{\phi^{*}(-\mathfrak{L}_\phi(x)) + \phi(x) \phi^*(\mathfrak{L}_\phi(x))}{\phi^{*}(\mathfrak{L}_\phi(x)) -  \phi(x) \phi^*(\mathfrak{L}_\phi(x))} = \frac{\frac{\phi^{*}(-\mathfrak{L}_\phi(x))}{\phi^*(\mathfrak{L}_\phi(x))} + \phi(x) }{1 -  \phi(x)}\geq  \frac{1 +  \phi(x)}{1 -  \phi(x)}.
\end{align*}
	Consequently $\lambda_{F}(x)\to +\infty$ as  $\phi(x)\rightarrow 1^-$, and hence $\lambda_F(\Omega)=+\infty$. Moreover, \[
	\lambda_{F}(x) \leq \sup_{\xi \in T^*_x \Omega \setminus \{ 0 \}} \frac{\frac{\phi^{*}(-\xi)}{\phi^*(\xi)} + \phi(x) }{1 -  \phi(x)}
	= \frac{\sup_{\xi \in T^*_x \Omega \setminus \{ 0 \}} \frac{\phi^{*}(-\xi)}{\phi^*(\xi)} + \phi(x) }{1 -  \phi(x)}
	= \frac{\lambda_{\phi} + \phi(x) }{1 -  \phi(x)},
	\]
	which finishes the proof.
\end{proof}


\begin{proof}[Proof of Theorem \ref{theorem-final}] By Proposition \ref{Funkdist-0}/(iii),  $(\Omega,F)$ is forward complete.
Furthermore, it is established by  Shen \cite{ShenSpray} that $(\Omega,F)$ is a Cartan--Hadamard manifold satisfying:
	\begin{enumerate}[\rm (i)]

		\item\label{ShenSpray11} constant flag curvature $\mathbf{K}=-\frac14;$

		\item constant $S$-curvature  $\mathbf{S}=\frac{n+1}{2};$

		\item\label{ShenSpray22}  $\dm_{BH}=\sigma {\dd}x$, where
		\begin{equation}\label{Vol_Funk}
			\sigma  \equiv \frac{\vol(\mathbb{B}^n)}{\vol(\Omega)}=\frac{\omega_n}{\vol(\Omega)}.
		\end{equation}
	\end{enumerate}
Thus, in the assumptions of Theorems \ref{mainThm1}-\ref{Spositivenoncom}, we may take
\[ k=\frac{1}{2}\ \ {\rm  and}\ \ h=\frac{n+1}{2(n-1)}\]
  which ensures the validity of both \eqref{ric-s-condition} and \eqref{second-cond}.
\end{proof}

\begin{remark} In the special case of Funk metrics, due to \eqref{Funkd0xdx0}, the main test function in the proof of Theorem \ref{mainThm1} reduces to
	 $$u_\alpha(x)=-e^{-\alpha d_F(\mathbf{0},x)}=-(1-\phi(x))^\alpha,\ \alpha>0.$$
	This explicit form of  $u_\alpha$ as  well as the density $\sigma$ in \eqref{Vol_Funk} of the measure $\m_{BH}$ on $(\Omega,F)$ allow a direct, alternative verification  of the failure of functional inequalities in Theorem \ref{mainThm1}.   Such approach was performed in Krist\'aly \cite{Kris} to prove the vanishing of the first eigenvalue of the Laplacian in the special case when $\Omega=\mathbb B^n$ and $\phi(x)=|x|$, in which case the Funk metric has the explicit form appearing in \eqref{specialfunk}.

\end{remark}

This subsection concludes with a discussion of the interpolation metric introduced in Krist\'aly--Rudas\cite{KR}.
 For $a\in [0,1]$, define $F_a:T\mathbb{B}^n\to [0,\infty)$  by
\begin{equation}\label{specialfunk-interpolation}
	F_a(x,y)
	=\frac{\sqrt{ |y |^2-|x|^2|y|^2+\langle x,y \rangle^2} +a\langle x,y\rangle}{1-|x|^2},
	\quad x \in \mathbb{B}^n,\ y \in T_x\mathbb{B}^n=\mathbb{R}^n.
\end{equation}
This family interpolates between two fundamental metrics:
the Funk metric \eqref{specialfunk} for $a=1$  with ${\bf K}=-1/4$ and ${\bf S}_a=(n+1)/2$,  and the Hilbert/Klein metric for $a=0$ with ${\bf K}=-1$  and ${\bf S}_a=0$. Here ${\bf S}_a$ denotes the $S$-curvature of the Busemann--Hausdorff measure $\m_a$.

The space $(\mathbb B^n,F_a)$ is a projectively flat Finsler space of Randers-type. Its  reversibility is given by $\lambda_{F_a}(\mathbb B^n)=\frac{1+a}{1-a}$ for $a\in [0,1)$, while $\lambda_{F_1}(\mathbb B^n)=+\infty$.   Furthermore, in Kaj\'ant\'o--Krist\'aly
\cite[Theorems 1.1 and 1.2]{KK-JGA}, it is proved that for $a\in (0,1)$, up to an isometry, the flag curvature and $S$-curvature  of $F_a$ satisfy the sharp estimates
\[-\frac{1}{(1-a)^2}<{\bf K}_a<-\frac{1}{(1+a)^2}\ \ {\rm and} \ \ 0<{\bf S}_a<\frac{(n+1)a}{2(1-a^2)}.\]
In particular, these estimates are \textit{sharp} for every $a\in (0,1)$. In fact, for every $a\in [0,1)$ one has that $\inf_{(x,y)\in T\mathbb B^n}{\bf S}_a(x,y)=0$, and hence Theorems \ref{reversibinfite} and \ref{Spositivenoncom} do not apply (i.e., $0=h<k$).

Nevertheless, by combining Theorem \ref{theorem-final} with elementary estimates as in \cite{KR} and the validity of famous functional inequalities on Riemannian Cartan--Hadamard manifolds \cite{HKZ,K2}, we obtain the following sharp criterion:
\begin{proposition}
Let $a\in [0,1]$, and consider the interpolated Funk-type space $(\mathbb B^n,F_a,\m_a)$ endowed with its natural Busemann--Hausdorff measure $\m_a$. Then the Hardy  inequality, Heisenberg--Pauli--Weyl uncertainty principle, and  Caffarelli--Kohn--Nirenberg inequality hold  on $(\mathbb B^n,F_a,\m_a)$ -- with possibly not sharp constants --  if and only if $a\in [0,1)$.
\end{proposition}


%

\subsection{Counterexamples under $\mathbf{Ric}_N \geq -K^2$}\label{Counterexamples}
This subsection demonstrates the necessity of our separate curvature assumptions
\begin{equation}\label{separetedcondition}
\mathbf{Ric}\geq -(n-1) k^2,\quad \mathbf{S} \geq (n-1){h}
\end{equation}
by showing that the main results \textit{fail} if those are replaced by a lower bound on the weighted Ricci curvature $\mathbf{Ric}_N$.

 Recall that for a Finsler metric measure manifold $(M,F,\mathfrak{m})$,  the weighted Ricci curvature $\mathbf{Ric}_N$
combines both $\mathbf{Ric}$ and $\mathbf{S}$, which is a Finslerian version of the Bakry--\'Emery Ricci tensor. The lower bound $\mathbf{Ric}_N \geq K$ is equivalent to the $\mathrm{CD}(K,N)$ condition (see Ohta \cite{Ohtabook,OS,Ohta-pacific}).

More precisely, the
  {\it weighted Ricci curvature} $\mathbf{Ric}_N$  is defined as follows: given $N\in [n,\infty]$, for any unit vector $y\in SM$,
\begin{align}\label{defRicN}
\mathbf{Ric}_N(y)=\left\{
\begin{array}{lll}
\mathbf{Ric}(y)+\left.\frac{\dd}{{\dd}t}\right|_{t=0}\mathbf{S}(\gamma'_y(t))-\frac{\mathbf{S}^2(y)}{N-n}, && \text{ for }N\in (n, \infty),\\
\\
\underset{L\downarrow n}{\lim}\mathbf{Ric}_L(y), && \text{ for }N=n, \\
\\
\mathbf{Ric}(y)+\left.\frac{\dd}{{\dd}t}\right|_{t=0}\mathbf{S}(\gamma'_y(t)),  && \text{ for }N=\infty.
\end{array}
\right.
\end{align}
Although it seems  natural to replace  the separated condition \eqref{separetedcondition} by $\mathbf{Ric}_N\geq -K^2$ for some $K\geq 0$, the following examples demonstrate that Theorems \ref{mainThm1} and \ref{reversibinfite} will fail.

\begin{example}{\rm (Li \cite{Li})}\label{reversminkspaces} There are infinitely many reversible Minkowski norms  on Euclidean spaces. In the sequel, we provide two such classes:
\begin{itemize}
\item [(a)]
Given $m,n \in \mathbb{Z}^+$, then
\[
\phi(y) = \sqrt[2m]{\sum^n_{i=1} (y^i)^{2m}},\quad \forall\,y=(y_1,...,y_n)\in \mathbb R^n,
\]
is a reversible Minkowski norm on $\mathbb{R}^n$.

\item[(b)] Let $|y|$ and $|\tilde{y}|$ be the Euclidean norms on $\mathbb{R}^n$ and $\mathbb{R}^m$, respectively.
Then
\[
\phi(y,\tilde y) = \sqrt{\sqrt{|y|^4 + |\tilde{y}|^4} + |y|^2 + |\tilde{y}|^2},\quad \forall\, (y,\tilde y)\in \mathbb{R}^n \times \mathbb{R}^m,
\]
is a reversible Minkowski norm on $\mathbb{R}^n \times \mathbb{R}^m$.

\end{itemize}
In general, given  a reversible Minkowski norm $\phi$ on the Euclidean space $\mathbb{R}^n$,  the triple $(\mathbb{R}^n,\phi,\mathscr{L}^n)$ is a complete noncompact $\FMMM$ with $\mathbf{Ric}_N\equiv 0$ for all $n\in [N,\infty]$. However, according to \cite{HKZ}, the $L^p$-Hardy inequality,   Heisenberg--Pauli--Weyl uncertainty principle and  Caffarelli--Kohn--Nirenberg interpolation
inequality are valid as well as sharp with the same optimal constants as in the Euclidean case; we also notice that $\mathbf{S}\equiv0$ for such spaces.
\end{example}

\begin{example}[Hilbert/Klein metric]\label{hilberexample}
Let $F = F(x,y)$ be a Funk metric on a bounded strongly convex $\Omega\subset \mathbb{R}^n$, see Definition \ref{funkdef}.
The corresponding {\it Hilbert metric} (or {\it Klein metric}) is given as
\begin{equation}
\tilde{F} (x,y)=  \frac{1}{2} \left\{ F(x,y) + F(x, -y)  \right\}.
\end{equation}
Consider the triplet $(\Omega,\tilde{F},\m_{BH})$, where $\m_{BH}$ is the Busemann--Hausdorff measure  of $F$ (not $\tilde{F}$), which is a constant multiple of Lebesgue measure (see \eqref{Vol_Funk}). Note that $\tilde{F}$ is  reversible, i.e., $\lambda_{\tilde{F}}(\Omega) = 1$.
Then, by \cite{ShenSpray},   $(\Omega,\tilde{F},\m_{BH})$ is an  $n$-dimensional  complete {\it reversible} $\FMMM$  of constant flag curvature $\widetilde{\mathbf{K}} = -1$. Additionally, Ohta \cite{Ohta-pacific} shows that
\begin{equation}\label{weighRicc}
	\widetilde{\mathbf{Ric}}_N \geq - (n - 1) -  \frac{(n+1)^2}{N-n}, \quad \forall\,N\in (n,\infty].
\end{equation}
These properties demonstrate that in Theorem \ref{reversibinfite}, the  curvature conditions \eqref{separetedcondition} cannot be replaced by a lower bound on the weighted Ricci curvature alone.
\end{example}

\begin{remark}\rm
	We can provide a short, alternative proof to \eqref{weighRicc}, by using our previous arguments combined with some well-known facts about Hilbert metrics.
	
	First, we observe that $\widetilde{\mathbf{Ric}} = - (n-1)$, as $\widetilde{\mathbf{K}} = -1$.  Moreover, we claim
	\begin{equation}\label{HilbertSrange}
		-(n+1) \leq  \tilde{\mathbf{S}}\leq  n+1.
	\end{equation}
Indeed, according to Okada \cite{Ok}, every Funk metric satisfies
\begin{equation}\label{FunkEquation}
	\frac{\partial F}{\partial x^k} = F \frac{\partial F}{\partial y^k} .
\end{equation} Let $P=P(x,y)$ denote the projective factor of the Funk metric $F$.
Then \eqref{prodef} and \eqref{FunkEquation} yield
\[
P = \frac{y^k}{2 F}\frac{\partial F}{\partial x^k} = \frac{y^k}{2}\frac{\partial F}{\partial y^k} =\frac{1}{2} F .
\]
Then the projective factor of the Hilbert metric  $\tilde{F}$ is given by
\begin{align*}
	\tilde{P}(x,y) & =  \frac{y^k\tilde{F}_{x^k}(x,y)}{2 \tilde{F}(x,y)} = \frac{  P(x,y) F(x,y) -  P(x, -y) F(x, -y) }{  F(x,y) + F(x, -y) }\\
	& = \frac{ F^2(x,y) - F^2(x, -y) }{ 2\left[ F(x,y) + F(x, -y)\right] } = \frac{1}{2}\left[ F(x,y) - F(x, -y)\right],
\end{align*}
which implies that for every $y\neq 0$, we have
\begin{equation}\label{HilbertPrange}
	\frac{\tilde{P}^2(x,y)}{\tilde{F}^2(x,y)} = \frac{[F(x,y) - F(x, -y)]^2}{[F(x,y) + F(x, -y)]^2}\in [0,1).
\end{equation}
Consider the Busemann--Hausdorff measure $\m_{BH}$ of $F$, whose density function $\sigma(x)$ is a constant, see \eqref{Vol_Funk}.
Since $\tilde{G}^m = \tilde{P} y^m$,   relations \eqref{Scurvature} and \eqref{pohom} yield
\begin{align}\label{HilbertS}
	\tilde{\mathbf{S}}(x,y) &= \frac{\partial \tilde{ G}^m(x,y)}{\partial y^m} - y^m \frac{\partial }{\partial x^m} (\ln \sigma(x))= y^m\frac{\partial \tilde{P}(x,y)}{\partial y^m}+n \tilde{P}(x,y)=({n+1}) \tilde{P}(x,y),
\end{align}
which combined with \eqref{HilbertPrange} yields \eqref{HilbertSrange}.

Finally,   according to Troyanov \cite[p.\ 103]{Troyanov}, we have
\begin{equation}\label{troy}
	\frac{1}{\tilde{F}^2}\left( \tilde{P}^2 - \frac{\partial \tilde{P}}{\partial x^k} y^k \right)=\tilde{\mathbf{K}}  = -1.
\end{equation}
Then for any $y\in \widetilde{S\Omega}:=\{v\in T\Omega\,:\, \tilde{F}(v)=1 \}$, by \eqref{geodesequ}, \eqref{HilbertS}, \eqref{pohom} and \eqref{troy}, we obtain
\begin{align} \label{HilbertSde}
	\left.\frac{\dd}{{\dd}t}\right|_{t=0}\tilde{\mathbf{S}}({\gamma}'_y(t))&
	=  \frac{\partial \tilde{\mathbf{S}}}{\partial x^k} y^k -2 \frac{\partial \tilde{\mathbf{S}}}{\partial y^m}  \tilde G^m(y)= (n+1) \left( \frac{\partial \tilde{P}}{\partial x^k} y^k - 2 \frac{\partial \tilde{{P}}}{\partial y^m}  \tilde G^m(y) \right) \notag\\
	&= (n+1) \left( \frac{\partial \tilde{P}}{\partial x^k} y^k - 2 \tilde P \frac{\partial \tilde{{P}}}{\partial y^m}  y^m \right) \notag\\
	& = (n+1) \left( \frac{\partial \tilde{P}}{\partial x^k} y^k - 2 \tilde{P}^2 \right)= (n+1) (\tilde{F}^2 - \tilde{P}^2).
\end{align}
Therefore, it follows by \eqref{defRicN},  \eqref{HilbertS} and \eqref{HilbertSde} that for any unit vector $y\in \widetilde{S\Omega}$,
\begin{align*}\label{HilbertRicN}
	\widetilde{\mathbf{Ric}}_N(y)=\left\{
	\begin{array}{lll}
		2  - (n+1)\tilde{P}^2 -  \frac{(n+1)^2}{N-n} \tilde{P}^2, && \text{ for }N\in (n,\infty),\\
		\\
		2 - (n+1)\tilde{P}^2,  && \text{ for }N=\infty,
	\end{array}
	\right.
\end{align*}
which combined with relation \eqref{HilbertPrange} implies the required estimate \eqref{weighRicc}.

\end{remark}

\end{document}